\documentclass[11pt]{amsart}
\usepackage{a4wide}
\usepackage{amsmath}
\usepackage{tikz}
\usepackage{amsthm}
\usepackage{hyperref}
\usepackage{amssymb}
\usepackage{amscd}
\usepackage[all]{xy}

\newtheoremstyle{theoremstyle}
  {10pt}      
  {5pt}       
  {\itshape}  
  {}          
  {\bfseries} 
  {:}         
  {.5em}      
  {}          

\newtheoremstyle{examplestyle}
  {10pt}      
  {5pt}       
  {}          
  {}          
  {\bfseries} 
  {:}         
  {.5em}      
  {}          

\theoremstyle{plain}
\newtheorem{theorem}{Theorem}[section]
\newtheorem{lemma}[theorem]{Lemma}
\newtheorem{proposition}[theorem]{Proposition}
\newtheorem{corollary}[theorem]{Corollary}

\theoremstyle{definition}
\newtheorem{definition}[theorem]{Definition}
\newtheorem{example}[theorem]{Example}

\newtheorem{notation}[theorem]{Notation}
\theoremstyle{remark}
\newtheorem{remark}[theorem]{Remark}

\newcommand{\bigslant}[2]{{\ensuremath\raisebox{.2em}{$#1$}\left/\raisebox{-.2em}{$#2$}\right.}}
\newcommand{\dd}{\mbox{d}}
\newcommand{\Loc}{\mbox{Loc}}
\newcommand{\Sp}{\mbox{Sp}}
\newcommand{\Sps}{\mbox{\tiny Sp}}
\newcommand{\cG}{\mathcal{G}}
\newcommand{\T}{\mathbb{T}}
\newcommand{\Z}{\mathbb{Z}}
\newcommand{\C}{\mathbb{C}}

\newcommand{\D}{\mathbb{D}}
\newcommand{\R}{{\mathbb{R}}}
\newcommand{\N}{{\mathbb{N}}}
\newcommand{\Hom}{{\textrm{Hom}}}

\newcommand{\Id}{\mbox{Id}}
\renewcommand{\phi}{\varphi}

\newcommand{\WW}{\mathbb{W}}
\newcommand{\W}{\mathcal{W}}

\newcommand{\Ad}{\mbox{Ad}\hspace{0.1cm}}

\DeclareMathOperator*{\Ker}{\mbox{Ker}}
\DeclareMathOperator*{\End}{\mbox{End}}

\DeclareMathOperator*{\bigstarm}{\bigstar}
\newcommand{\AH}{{\mathbb{A}_\hbar}}

\newcommand{\hhbar}{[\![\hbar]\!]}
\newcommand{\chec}{\check{\mathrm{H}}}
\newcommand{\CH}{\mbox{H}}
\newcommand{\Ima}{\mbox{Im}\hspace{0.1cm}}
\newcommand{\norm}[1]{\left\lVert#1\right\rVert}
\newcommand{\umega}{\bar{\omega}}
\newcommand{\Aut}{\mbox{Aut}}
\newcommand{\Symp}{\mbox{Symp}}

\newcommand{\TF}{\mathcal{T}_\mathcal{F}}
\newcommand{\g}{\mathfrak{g}}
\newcommand{\lsp}{\mathfrak{sp}}
\newcommand{\tg}{\tilde{\mathfrak{g}}}

\newcommand{\tM}{\widetilde{M}}
\newcommand{\G}{{\mathcal{G}_\nabla}}
\newcommand{\GZ}{{\overline{\mathcal{G}_\nabla}}}
\newcommand{\cZ}{\mathcal{Z}}
\newcommand{\Cht}{\C\hhbar^\times}
\newcommand{\Ch}{\C\hhbar}
\title[Group actions on symplectic deformation quantization]{Extension and Classification of Group Actions on \\Formal Deformation Quantizations \\of Symplectic Manifolds}
\author{Niek de Kleijn}
\thanks{Niek de Kleijn was supported by the Danish National Research Foundation
through the Centre for Symmetry and Deformation (DNRF92).}
\begin{document}
	\begin{abstract}
	We examine the existence and classification problems for group actions on deformation quantizations of symplectic manifolds. We restrict ourselves to the star products obtained 
	using the Fedosov construction on a symplectic manifold $(M,\omega)$. We define a notion of lift of a group action that does not impose restrictions on the Fedosov connection realizing a certain gauge equivalence class of star products. We reformulate the well known sufficient conditions for existence of such lifts and prove (by way of example) that they are not necessary. Given existence of a lift we define a (non-Abelian) group $\GZ$ such that the first group cohomology with values in $\GZ$ classifies lifts up to a natural notion of equivalence. Finally we develop tools that allow for computation of this group cohomology and apply these to do computations in a range of examples. 
\end{abstract}
\maketitle

	\section{Introduction}
	\subsection{Background} 
	
	The concept of deformation quantization and in particular formal deformation quantization was famously introduced in the influential paper \cite{BFFLS} and has been extensively studied 
	since then. It phrases the problem of quantization of physical systems in terms of the mathematical notion of formal deformations of algebras developed by Gerstenhaber in the series of 
	papers 
	\cite{G}. One of the main achievements in the field is of course Kontsevich's classification of the formal deformation quantizations of Poisson manifolds \cite{Kont}. It is however 
	relevant to note that Fedosov had essentially performed such a classification in the case of symplectic manifolds \cite{Fb, Fp} (which are arguably the ones most encountered in 
	physics) earlier.
	
	\vspace{0.3cm}
	
	One of the advantages noted in Fedosov's original paper is that there is a simple way to lift any symplectomorphism of the manifold to an automorphism of the deformation 
	quantization. It pays to keep in mind however that these lifts do not obviously respect compositions and are in general not unique. Since 
	group actions on the symplectic manifold can be interpreted as symmetries of the classical mechanical system, the question of whether one can lift a group action naturally arises and 
	Fedosov comments on this in his original paper \cite{Fp}. 
	
	\vspace{0.3cm}
	
	Another reason why this question arises is in the context of the algebraic index theorem \cite{N-T95, Fp}. The main observation underlying this theorem is that the pseudo-differential 
	operators on a smooth manifold form a deformation quantization of the cotangent bundle of that manifold (with the canonical symplectic structure). 
	Given this observation the algebraic index theorem 
	is the algebraic counterpart of the Atiyah-Singer index theorem \cite{A-S}. Note that the algebraic index theorem holds for any deformation quantization of any symplectic manifold however
	(not just cotangent bundles and pseudo-differential operators). It was also shown by Nest-Tsygan 
	in \cite{N-T96} that one can rederive the (analytic) index theorem from the algebraic one. If we now consider a manifold carrying a group action (or a foliation) it leads to the equivariant 
	(resp transversal) index theorem \cite{C-M, P-R}. Therefore it makes sense to expect algebraic counterparts to these theorems as well. This makes even more sense since these counterparts 
	were found for orbifolds \cite{PPT2}. In order to derive such an equivariant algebraic index theorem in general, we should first understand the nature of induced group actions on deformation 
	quantizations and this paper addresses that issue. 
	
	\vspace{0.3cm}
	
	We should note here that this paper has been part of the thesis work \cite{thesis}, in which the relation to the equivariant index theorem is clarified and many parts of this paper are reproduced. This paper was in fact produced, but not peer reviewed, prior to the production of \cite{thesis} and it seems appropriate in any case to offer a stand-alone account of the results. Similarly to this paper one can find a stand-alone account of the equivariant algebraic index theorem in \cite{EQAIT}.
	
	\vspace{0.3cm}
	
	Some work towards both the classification and existence results of such group actions on formal deformation quantizations of symplectic manifolds was in fact already carried out. 
	First of all Fedosov comments on the problem in both his original paper and his book \cite{Fp, Fb}. Furthermore, when one restricts to so called invariant star products (a definition will 
	follow),  a classification up to equivariant equivalence was found by Bertelson-Bieliavsky-Gutt 
	\cite{Ginvariant} and recently this classification was extended to include a notion of quantum moment maps by Reichert-Waldmann \cite{RW}. In this paper we will extend these considerations to provide some 
	(limited) existence results and (extensive) classification results for group actions on symplectic 
	manifolds, where we consider any  extensions to the deformation quantization. The reason to consider more than the natural extensions is explained in Remark \ref{why}.
\subsection{Set-up, Notation and Convention} 
Let us now be more precise about the concepts mentioned above and at the same time introduce the notations and conventions used in this paper. 

\vspace{0.3cm}

First of all we will consider only symplectic manifolds $(M,\omega)$ of dimension $2d$. We will denote the corresponding 
Poisson bracket on the algebra $C^\infty(M)$ of complex valued smooth functions on $M$ by $\{\cdot,\cdot\}$. Note that we will always consider complex valued functions and thus also complex valued cohomology. 
Furthermore let us fix the formal variable $\hbar$ for the rest of this paper. Vector spaces and linear maps will often be considered as defined over the ring $\C\hhbar$ of formal power series with coefficients in $\C$. 
\begin{definition}\label{defdefq}[Formal Deformation Quantization]
	A formal deformation quantization of a symplectic manifold $(M,\omega)$ is given by an associative product $\star$ on the formal power series in $\hbar$ with coefficients in $C^\infty(M)$
	such that 
	\begin{enumerate} 
		\item $\star$ is a deformation of the $\C\hhbar$-linear extension of the usual product, i.e. \\$\forall f,g\in C^\infty(M)\hhbar$ we have that 
		\[f\star g=f\cdot g +\sum_{k\geq 1}\hbar^kB_k(f,g)\] for some $\C\hhbar$-bilinear operators $B_k$; 
		\item $\star$ is local, i.e. the operators $B_k$ are bidifferential; 
		\item $\star$ is a quantization of $\{\cdot,\cdot\}$, i.e. $\forall f,g\in C^\infty(M)\hhbar$ 
		\[ f\star g-g\star f = i\hbar\{f,g\} +\hbar^2 C(f,g),\]
		where we extend $\{\cdot,\cdot\}$ to $C^\infty(M)\hhbar$ by $\C\hhbar$-bilinearity and $C$ is a $\C\hhbar$-bidifferential operator, and 
		\item $\star$ is normalized, i.e.  we have $1\star f=f=f\star 1$ for all $f\in C^\infty(M)\hhbar$. 
		
	\end{enumerate}
	
\end{definition}

Note that the locality assumption implies that the algebra $(C^\infty(M)\hhbar,\star)$ is actually given by the global sections of a sheaf of algebras $\AH$. In the following we will often denote 
a deformation quantization by $\AH$ or $\AH(M)$ if we want to emphasize especially that we are dealing with a sheaf. We will call $\AH$ a formal deformation quantization, deformation 
quantization or sometimes even a deformation, since we are not considering any other deformation quantizations or even deformations this should not lead to any confusion. We will rarely explicitly 
write the product as $\star$. The product used on different algebras should generally be clear from context. 

\vspace{0.3cm} 

The Fedosov construction \cite{Fb, Fp}, which we will give a more elaborate recollection of in the next section,  is a method for constructing deformation quantizations of symplectic manifolds. In this construction one starts by considering a certain bundle of algebras 
$\mathcal{W}\rightarrow M$ called the Weyl bundle. This bundle is straightforwardly constructed from the bundle of symplectic vector spaces $(T^*M,\umega)$, where we denote by 
$\umega_x$ the symplectic structure on $T^*_xM$ induced by $\omega_x$ and the isomorphism $T^*_xM\simeq T_xM$ given by $v\mapsto \omega_x(v,-)$. One then defines a certain connection, 
usually called a \emph{Fedosov} connection $\nabla$ on $\mathcal{W}$ and shows that the kernel $\Ker \nabla \subset \Gamma(\W)$ (we will often use $\Gamma(-)$ to refer to smooth sections) comes 
equipped with a linear isomorphism to $C^\infty(M)\hhbar$ such that the induced product $\star$ defines a deformation quantization. We will also present a re-interpretation of the construction through the lens of formal geometry as in \cite{N-T95}. 

\begin{definition} 
	A \emph{gauge equivalence} between two deformation quantizations ${\AH_1}$ and ${\AH}_2$ is given by a series of $\C\hhbar$-linear operators $\{T_k\}_{k\geq 1}$ such that 
	\[\Id+\sum_{k\geq 1}\hbar^kT_k\colon {\AH}_1(M)\longrightarrow {\AH}_2(M)\] 
	is an algebra isomorphism. The operators $T_k$ are then necessarily differential.
\end{definition}

In fact it has been shown that deformation quantizations of $(M,\omega)$ are classified up to gauge equivalence by their (Deligne) characteristic class \cite{Deligne} in $\CH^2(M)\hhbar$ 
(formal power series with coefficients in $\CH^2(M,\C)$). Here we mean de Rham cohomology by $\CH^\bullet(M;A)$ ($A$ refers to coefficients) as will we in the rest of this paper. 
We will denote differential forms by the symbol $\Omega$, i.e. $\Omega^3(M;E)$ will mean differential $3$-forms with values in the bundle $E$. Given a closed $2$-form 
$\theta\in\Omega^2(M)\hhbar$ one can construct a Fedosov connection $\nabla_\theta$ such that the corresponding deformation quantization has characteristic class 
$[\theta]\in \CH^2(M)\hhbar$ \emph{uniquely} (we will always use $[\cdot]$ to signify equivalence classes). So we see that up to gauge equivalence every deformation quantization is a Fedosov quantization. 

\vspace{0.3cm} 

We show in Lemma \ref{lifts} that there is a natural \emph{surjection} 
\[\Aut(\AH)\longrightarrow \Symp(M,\omega)_{[\theta]}.\] 
Here we have denoted  the automorphisms of the sheaf of algebras $\AH$ by $\Aut(\AH)$ and  the stabilizer of the characteristic class $[\theta]$ 
of $\AH$,  for  the natural action of symplectomorphisms $\Symp(M,\omega)$ on $\CH^2(M)\hhbar$, by $\Symp(M,\omega)_{[\theta]}$. Thus we can extend any symplectomorphism of $M$ that stabilizes the characteristic class 
to the deformation quantization. The map is such that given a symplectomorphism $\phi\in\Symp(M,\omega)_{[\theta]}$ an element $a_\phi\in\Aut(\AH)$ in the fiber over $\phi$ 
acts in lowest order in $\hbar$ like $\phi^*$ (we denote pull-backs by upper stars and push-forwards by lower stars).  In \cite{N-T99} the authors show that, given any such $\phi$, 
one can construct an invertible element $U_\phi\in\Gamma(\mathcal{W})^\times$ (we will denote invertibles by a superscript $\times$) such that 
\[\Aut(\AH)\ni\Ad U_\phi\circ\phi^*\mapsto \phi\in\Symp(M,\omega)_{[\theta]}\] 
where we denote the map given by $f\mapsto UfU^{-1}$ by $\Ad U$. In fact they construct $\delta_\phi\in\Gamma(\mathcal{W})$ such that $U_\phi=e^{\delta_\phi}$. Note that since any gauge equivalence is locally 
given by an inner automorphism \cite{Fp} we find that lifts are more generally of the form given above. This will reflect in our definition \eqref{extendaction} of extensions of 
group actions to deformation quantizations. 

\vspace{0.3cm} 

The paper is structured in the following way. In section 2 we recall the main constructions and results of the Fedosov construction. Then, in section 3,  we establish some of the facts presented above. 
In particular we define what we mean by an extension of an action by symplectomorphisms to the deformation. We proceed, in section 4, by giving first an abstract and generally hard to check requirement 
that needs to 
be fulfilled in order for a given action to extend. In Proposition \ref{trivialextensions} we give a few more easily verifiable conditions under which extension of the group action is 
possible. 
However, we also give an example 
of an action that extends, but does not satisfy these conditions. The main part, section 5, of the paper deals with the classification of such extended group actions. In this part we define a 
certain set $\CH^1(\Gamma,\GZ)$ which classifies extensions of the action of $\Gamma$ by symplectomorphisms to the deformation quantization corresponding to the Fedosov connection $\nabla$. 
Lastly we derive certain computational tools (summarized in the diagram \eqref{computingsquare}) that allow us to show that this set $\CH^1(\Gamma,\GZ)$ is meaningful in the sense that it 
allows to be computed in a variety of cases. 

\section{The Fedosov Construction} 

	In this section we will give a brief reminder of the Fedosov construction for symplectic manifolds, see \cite{Fp, Fb, N-T99, W} for more in depth expositions of the construction. We will also present the viewpoint of this construction when one considers it in the framework of formal geometry as is done in, for instance, \cite{N-T95} and \cite{Felder}. 

\subsection{Construction}

	The Fedosov construction allows one to construct any 
deformation quantization of a symplectic manifold (up to gauge equivalence) as a subalgebra of a certain algebra of sections. Let us first show how one can construct the underlying bundle of 
algebras, the so called \emph{Weyl  algebras bundle}.

\begin{definition} 
	
 Let $W\colon SymV_\C\longrightarrow GrA_{\C\hhbar}$ denote the Weyl functor from the category of complex symplectic vector spaces to the category of graded associative algebras over $\C\hhbar$ given by
	\[W\left(V,\alpha\right)=\widehat{T(V^*)\hhbar/I_\alpha},\] 
	 $V^*$ is the dual to $V$, $T(V^*)$ denotes the tensor algebra of $V^*$, the ideal $I_\alpha\subset T(V^*)\hhbar$ is generated by the elements 
	\[v\otimes w-w\otimes v-i\hbar\overline{\alpha}\left(v,w\right)\] for $v,w\in V^*$ and the hat signifies completion in the 
	$\langle V\rangle$-adic topology. Here $\overline{\alpha}$ denotes the symplectic structure dual to $\alpha$. The grading is given 
	by the assertion that $|\hbar|=2$ and $|v|=1$ for all $v\in V$. 
\end{definition}

\begin{notation}
	Given $\R^{2m}$ ($m\in\N$) with the standard symplectic structure $\omega_{st}$ we construct the standard complex symplectic vector 
	space by considering $\C^{2m}=\R^{2m}\otimes_\R\C$ with the symplectic structure 
	$\eta_{st}=\omega_{st}\otimes_\R 1$. Then we will denote $\WW^m:=W\left(\C^{2m},\eta_{st}\right)$ (we will usually simply denote it  $\WW$ if the dimension is implied). In the following we will always denote the component of homogeneous elements of degree $k$ by a subscript, i.e. $\WW_k$ and so on.
\end{notation}

\begin{notation}
	Given a symplectic manifold $M$ of dimension $2d$ we shall denote the principal $\Sp(2d,\R)$-bundle of symplectic frames (in the cotangent bundle) by 
	$\pi_1\colon \Sp_M\rightarrow M$. 
	\end{notation}

	Note that, since $W$ is a functor, the standard action of $\Sp(2d,\R)$ 
on $\R^{2d}$ induces an action on $\WW$.

	\begin{definition}\label{Weylalgebrasbundle}[Weyl Algebras Bundle] 
 We denote the bundle of algebras over $M$ associated to $Sp_M$ with fiber
	$\WW$ by $\W$.  We will call $\W$ the Weyl algebras bundle over $M$. 
\end{definition}

Note that 
\[\W_x=W\left(T^*_xM\otimes_\R\C,\umega_x\otimes_\R 1\right).\] We obtain the subalgebra isomorphic to a formal deformation quantization as the kernel of a connection on the Weyl algebras bundle.

\begin{definition} 
	By a connection on $\W$  we mean a $\C\hhbar$-linear map \[\nabla\colon\Omega^\bullet\left(M;\W\right)\longrightarrow \Omega^{\bullet+1}\left(M;\W\right)\] 
	such that 
	\[\nabla\left(\sigma\tau\right)=\nabla\left(\sigma\right)\tau+\left(-1\right)^k\sigma\nabla\left(\tau\right),\] for all
	$\sigma\in\Omega^k\left(M;\W\right)$ and $\tau\in\Omega^l\left(M;\W\right)$, and $\nabla f=\dd f$ for $f\in C^\infty\left(M\right)\subset \Gamma\left(\W\right)$. Here the inclusion comes from 
	the inclusion of the ground field in the tensor algebra.
\end{definition}

\begin{notation}
	We will denote the Lie algebra of continuous derivations of $\WW$ by $\g$ and we denote the group of continuous automorphisms of $\WW$ by $G$. Here we consider $\WW$ equipped with the topology induced by the grading. We shall also denote the Lie algebra given by $\frac{\WW}{i\hbar}$ equipped with the commutator bracket of $\WW$ by $\tg$. 
 \end{notation}

Note that the grading of $\WW$ yields the filtrations 
\[\g:=F_{-1}\g\supset F_0\g\supset\ldots\supset F_k\g\supset\ldots\]
and
\[G:=G_0\supset G_1\supset \ldots\supset G_k\supset \ldots\] and the grading 
\[\tg=\prod_{k=-2}^\infty \tg_k=\prod_{k=0}^\infty \frac{\WW_k}{i\hbar}.\]
Let us summarize some of the results about these in the following proposition, see 
\cite{RRII, N-T95} and \cite{thesis} for proofs of these results. 

\begin{proposition}\label{gG}
	
	\leavevmode
\begin{itemize} 
	\item The sequence 
	\begin{equation}\label{ses}0\rightarrow\frac{\C\hhbar}{i\hbar}\longrightarrow \tg\stackrel{P_\g}{\longrightarrow} \g\rightarrow 0\end{equation}
	
	given by $P_\g(\frac{X}{i\hbar})=\frac{1}{i\hbar}[X,-]$ is exact and preserves filtrations;
	\item \[\g=\prod_{k=-1}^\infty \g_k=\prod_{k=-1}^\infty P_\g(\tg_k),\] i.e. $\g$ is graded;
	\item $\g_{-1}\simeq \C^{2d}$ (with trivial Lie algebra structure);
	\item $\g_0\simeq \lsp(2d,\C)$ the Lie algebra of $\Sp(2d,\C)$ and $\bigslant{G}{G_1}\simeq \Sp(2d,\C)$;
	\item The sequence \eqref{ses} allows for an $\lsp(2d,\C)$ equivariant section $\tilde{\sigma}$;
	\item $F_0\g=F_1\g\rtimes \lsp(2d,\C)=P_\g\left(\frac{\WW_{k\geq3}}{i\hbar}\right)\rtimes \lsp(2d,\C)$;
	\item $F_1\g$ is a pro-nilpotent Lie algebra and we have $\exp(F_1\g)=G_1$. 
	\item $G\simeq G_1\rtimes \Sp(2d,\C)$.
\end{itemize}
\end{proposition}

\begin{notation} 
By keeping in mind the definition of $\WW$ through complexification of $\R^{2d}$ we obtain the subgroup $G_r=G_1\rtimes\Sp(2d,\R)\subset G$. 
\end{notation}

Note that, given a connection $\nabla$ on $\mathcal{W}$ which is continuous for the filtration topology on the fibers $\mathcal{W}_x$, we find by the first item of proposition \ref{gG}  that there always exists a $1$-form $B=B_{-1}+B_0+\ldots\in \Omega^1(M;\tg)$ such that, fixing a symplectic connection $\nabla_0$, we have
\[\nabla=\nabla_0+P_\g(B)=P_\g(B_{-1})+\nabla_0+P_\g(B_0)+P_\g(B_1)+\ldots.\]

\begin{definition} 
Let $A_{-1}\in\Omega^1(M;\tg_{-1})$ be given by 
$A_{-1}(X)=\frac{\omega(X,-)}{i\hbar}$ where we consider the obvious inclusion 
$\Omega^1(M)\hookrightarrow \Gamma(W)$ in degree $1$. 
\end{definition} 

\begin{definition}\label{Fedconn}[Fedosov connection] 
	A connection $\nabla$ on $\W$ is called a \emph{Fedosov connection} if it is \emph{flat}, i.e. $\nabla^2=0$, and of the form 
	\[\nabla=P_\g(A_{-1})+\nabla_0+P_\g(A_0)+P_\g(A_1)+P_\g(A_2)+\ldots\] for some symplectic  connection $\nabla_0$ and some 1-forms 
	$A_i\in \Omega^1(M;\tg_i)$ 
	\end{definition}

Note that the condition on a connection to be a Fedosov connection is simply that it has a specific lowest degree part and is flat. It turns out that torsion-freeness of the connection 
$\nabla_0$ allows us to set $A_0=0$ \cite{W, Fb}.  
	\begin{notation}\label{curv}
		Given 
	$A=A_{-1}+A_0+A_1+A_2+\ldots$ and $A_\N=A-A_{-1}$ with corresponding Fedosov connection $\nabla_A=\nabla_0+P_\g(A)$. We denote the curvature of the connection $\nabla_A$ by
	$\Omega_A:=\frac{R_0}{i\hbar}+\nabla_0\left(A_\N\right)+[A_{-1},A_\N]+\frac{1}{2}[A_\N,A_\N]$. Here $R_0$ is given by the curvature of $\nabla_0$.\end{notation} Note that for a Fedosov connection, since $\nabla_A^2=P_\g(\Omega_A)$, $\Omega_A$ 
is central and, since $\nabla_A\Omega_A=0$, also closed \cite{Fb}.

	\begin{theorem} 
	For all closed 2-forms $\theta\in\Omega^2\left(M,\C\right)\hhbar$ we can find a Fedosov connection $\nabla_\Theta$ such that the curvature $\Omega_\Theta=\frac{\omega}{i\hbar}+\theta$. 
\end{theorem}

See \cite{N-T99}, \cite{Fb}, \cite{W} for a proof of this theorem. Thus we see that  there is a surjective mapping $\nabla_A\mapsto \Omega_A-\frac{\omega}{i\hbar}$ from the set of Fedosov 
connections to $Z^2\left(M\right)\hhbar$ were $Z^2\left(M\right)$ denotes
the space of closed 2-forms. The above theorem sets up to classify all deformation quantizations of $C^\infty\left(M\right)$ (up to gauge equivalence) by $\CH^2\left(M\right)\hhbar$ by the following
propositions.
\begin{proposition}
	Suppose $\nabla$ is  a Fedosov connection. Then the complex $\left(\Omega^\bullet(M;\W),\nabla\right)$ is acyclic and we have a quasi-isomorphism of complexes 
	$Q\colon \left(C^\infty(M)\hhbar,0\right)\rightarrow \left(\Omega^\bullet(M;\W,\nabla\right)$. The algebra structure induced by the isomorphism 
	\[Q\colon C^\infty(M)\hhbar\longrightarrow \AH_\nabla:=\CH^0\left(\Omega(M;\W),\nabla\right)=\Ker \nabla\cap\Omega^0\left(M,\W\right)\] is a deformation quantization.
\end{proposition}

\begin{proposition}\label{Fedisall}
	Suppose $\AH\left(M\right)$ is a deformation quantization of $M$, then there exists a Fedosov connection $\nabla$ such that $\AH_\nabla\simeq \AH\left(M\right)$ by the map $Q^{-1}$ followed by a gauge equivalence $T$. 
\end{proposition}
A proof of this theorem may be found in \cite{N-T99}. The main argument is that one can extend the deformed product to jets of functions and so doing construct an isomorphism of jets with 
sections of $\W$. To finish the classification of deformation quantizations of $M$ we have the following theorem. 
\begin{theorem}\label{classification}
	Suppose $\nabla_A$, $\nabla_{A'}$ are two Fedosov connections, then $\AH_{\nabla_A}\simeq \AH_{\nabla_{A'}}$ if and only if 
	$[\Omega_A-\frac{\omega}{i\hbar}]=[\Omega_{A'}-\frac{\omega}{i\hbar}]\in \CH^2\left(M\right)\hhbar$. 
\end{theorem}
Again the proof may be found in \cite{N-T99} and the proof  of the ``if'' statement is given by constructing an automorphism of sections of the Weyl bundle degree by degree such that
conjugation by this automorphism maps 
one connection to the other. The main 
tool that is used is acyclicity of the complex with differential $[A_{-1},-]$. The ``only if'' statement follows again from the observation that one can extend the deformation to jets and 
the isomorphism implies an 
isomorphism of the corresponding algebras of jets. The above theorems combine to imply that $\CH^2\left(M\right)\hhbar$, or rather $\frac{\omega}{i\hbar}+\CH^2\left(M\right)\hhbar$, completely classifies formal deformation quantizations of the symplectic 
manifold $(M,\omega)$ up to gauge equivalence. In the following we will sometimes refer to the class in $\CH^2\left(M\right)\hhbar$ as the \emph{characteristic} class since the 
translation by $\frac{\omega}{i\hbar}$ will often not play a role.

\subsection{Elements of formal geometry} 

In this section we will comment (very briefly) on the Fedosov construction as considered in the framework of formal geometry, since we will need a few results from this framework. For a more in depth analysis see \cite{thesis} or \cite{EQAIT}. For this section we shall fix a deformation quantization $\AH(M)$ of $M$. 

\vspace{0.3cm} 

\begin{notation}
Denote the algebra of jets of functions at $m\in M$ by $J_m^\infty(M)$, explicitly we have 
$J^\infty_m(M):=\varprojlim C^\infty(M)/\left(\mathcal{I}_m\right)^k$ where $\mathcal{I}_m$ is
the ideal of functions vanishing at the point $m$ and $k\in \N$. We shall denote 
$J_0^\infty(\R^{2d})=:\mathbb{O}$.
\end{notation}

 Since the product $\star$ on $\AH(M)$ is local it also defines $\C\hhbar$-bilinear products $\star_m$ on the jets $\widehat{A_m(M)}:=J^\infty_m(M)\hhbar$ for all $m\in M$. 

\begin{notation} 
	We denote the algebra homomorphism $\AH(M)\rightarrow C^\infty(M)$ given by setting $\hbar=0$ by $\sigma$. Similarly we denote the induced algebra homomorphism 
	$\widehat{A_m(M)}\rightarrow J^\infty_m(M)$ by $\hat{\sigma}_m$. 
	\end{notation}

Consider the deformation quantization $\AH(\R^{2d})$ given by the Moyal--Weyl product 
\[
\left(f\star g\right)\left(\xi,x\right)=\exp\left.\left(\frac{i\hbar}{2}\sum_{i=1}^d(\partial_{\xi^i}\partial_{y^i}-\partial_{\eta^i}\partial_{x^i})\right)f
\left(\xi,x\right)g\left(\eta,y\right)\right|_{\substack{\xi^i=\eta^i\\x^i=y^i}}.
\]
Then we find that $\widehat{A_0(\R^{2d})}\simeq \WW$, see \cite{thesis, EQAIT}. Since symplectic manifolds are locally symplectomorphic to $\R^{2d}$, deformation quantization is local and deformation quantization of $\R^{2d}$ (with the standard symplectic structure) is unique up to gauge equivalence, we find that $\widehat{A_m(M)}\simeq \WW$ for all $m\in M$. 

\begin{definition} 
The manifold of deformed non-linear frames $\tM$ is given by 
\[\tM:=\left\{\phi_m\colon\widehat{A_m(M)}\stackrel{\sim}{\longrightarrow} \WW \right\},\]
where we consider only continuous isomorphisms. 
\end{definition}

Note that the decomposition of $G$ given in \ref{gG} shows that it has the structure of a pro-finite dimensional Lie group, since $G_1:=\exp(F_1\g)$ and $F_1\g$ is a pro-nilpotent Lie algebra. Thus we find the pro-finite dimensional manifold structure on $\tM$ given by considering it as a $G$-principal bundle over $M$ through the obvious action of $G$ and the projection $\phi_m\mapsto m$. For details we refer to \cite{thesis, N-T95}. 

\vspace{0.3cm} 

The construction of the manifold $\tM$ goes through mutatis mutandis in the case that we replace the deformation quantization $\AH(M)$ simply by the algebra of (real valued) smooth functions. In this case it can be interpreted as the manifold of jets of coordinate charts on $M$. If we want to carry through such an interpretation our current $\tM$ is slightly too large. 

\begin{definition} 
 Let $\tM_r\subset \tM$ denote the submanifold given by $\phi_m\in \tM_r$ if and only if 
 $\phi_m\in \tM$ and the unique map $\tilde{\phi}_m\colon J^\infty_m(M)\rightarrow \mathbb{O}$, 
 such that $\tilde{\phi}_m\circ \hat{\sigma}_m=\hat{\sigma}_0\circ\phi_m$, is induced by a symplectomorphism $S_\phi\colon \R^{2d}\rightarrow M$ sending $0\mapsto m$. 
\end{definition}

We note that $\tM_r$ is a $G_r$-principal bundle over $M$ and it is a $G_1$-principal bundle over $\Sp_M$. The projection $\tM_r\rightarrow \Sp_M$ is given by sending $\phi_m$ to the 
symplectic frame given by $\left\{(S_\phi)_*(\partial_{x_1}|_0), (S_\phi)_*(\partial_{\xi_1}|_0),\ldots, (S_\phi)_*(\partial_{\xi_d}|_0)\right\}$.

\begin{theorem}[Kazhdan connection] 
There is a natural isomorphism 
\[\omega_M(\phi_m)\colon T_{\phi_m}\tM\longrightarrow \g,\] for all $\phi_m\in\tM$, such that $\omega_M$ defines a 1-form in 
$\Omega^1\left(\tM\right)\otimes \g$ satisfying 
\begin{equation}\label{MC2}d\omega_M+\frac{1}{2}[\omega_M,\omega_M]=0,\end{equation} where $d$ denotes the exterior derivative.
\end{theorem}

See \cite{N-T95} for a proof. Now we note that, since $G_1\simeq F_1\g$ by the exponential map, there exists a section 
$\Sp_M\rightarrow \tM_r$. In fact one may prove that there exists an $\Sp(2d,\R)$ equivariant section. 

\begin{definition}
	Given an equivariant section $F\colon \Sp_M\rightarrow \tM_r$ we define the induced connection $\nabla_F$ on $\W$ as given by the connection form $(\iota\circ F)^*\omega_M\in \Omega^1(\Sp_M)\otimes \g$. Here $\iota\colon \tM_r\hookrightarrow \tM$ denotes the obvious inclusion. 
	\end{definition}

\begin{proposition} 
	Given an equivariant section $F\colon \Sp_M\rightarrow \tM_r$ the connection $\nabla_F$ is a Fedosov connection and the corresponding deformation quantization is gauge equivalent to $\AH(M)$.  
	\end{proposition} 

For the proof we refer again to \cite{thesis}. This concludes our summary of the relation between formal geometry and the Fedosov construction. 

\section{Automorphisms lifted from Symplectomorphisms} 

	In this section we will discuss the lifting problem associated to symplectomorphisms induced by automorphisms of deformation quantizations. Let 
$\nabla:=P_\g(A_{-1})+\nabla_0+P_\g(A_\N)$ be a fixed Fedosov connection and denote the corresponding deformation quantization by $\AH\left(M\right)$. If $\varphi$
is a $\C\hhbar$-linear algebra automorphism of $\AH\left(M\right)$ the assignment
$f\mapsto \varphi\left(f\right)\mod \hbar$ defines an algebra automorphism of $C^\infty\left(M\right)$. Thus there is a map 
\[\Aut\left(\AH\left(M\right)\right)\longrightarrow \Symp\left(M,\omega\right)^{op},\] given by noting that any automorphism of $C^\infty\left(M\right)$ is given (uniquely) by the pull-back 
under a diffeomorphism (see, for instance, Chapter 7 of \cite{JN}). The fact that the induced diffeomorphism is a symplectomorphism can be seen by applying it to a commutator.

\begin{definition} 
	Suppose $\varphi\in\Symp\left(M,\omega\right)$ we say $\alpha_\varphi\in \Aut\left(\AH\left(M\right)\right)$ lifts or extends $\varphi$ if 
	\[\alpha_\varphi\left(f\right)=\varphi^*f+\hbar \Phi_f,\]
	with $\Phi_f\in\AH$, for all $f\in C^\infty\left(M\right)\hhbar$. Here $\varphi^*$ denotes the $\C\hhbar$-linear extension of the pull-back $\varphi^*f=f\circ\varphi$.  
\end{definition}

Note that in this language a gauge equivalence of the deformation to itself is a lift of the identity. The following definition will be helpful.   

\begin{definition} 
	Note that $\Sp(2d,\R)$ acts on $G_1$ by conjugation. This yields a sheaf of groups $\cG_1$ given by the sections of $\Sp_M\times_{\Sps(2d)}G_1\rightarrow M$.
	Here $\Sp_M\times_{\Sps(2d)}G_1$ denotes the quotient of $\Sp_M\times G_1$ by the relation $(pg,\gamma)\sim(p,g\gamma g^{-1})$ for any $g\in\Sp(2d,\R)$. 
\end{definition}

The following proposition characterizes the symplectomorphisms that can be extended. 

\begin{proposition}\label{lifts}
	The map \[\Aut\left(\AH\left(M\right)\right)\longrightarrow \Symp\left(M,\omega\right)^{op}_{[\theta]}\] is a surjection. Here $[\theta]\in \CH^2(M)\hhbar$ denotes the 
	characteristic 
	class of $\AH\left(M\right)$ and\\ $\Symp\left(M,\omega\right)_{[\theta]}$ 
	denotes the stabilizer for the obvious (right) action on $\CH^2\left(M\right)\hhbar$. 
	
\end{proposition}
\begin{proof}
	
	\leavevmode
	
	\noindent Note that the lemma asserts that, for a symplectomorphism $\varphi$ to extend to the deformation, it is both a necessary and sufficient condition that it
	preserves the characteristic class of the deformation. So let us fix a symplectomorphism $\varphi$. 
	
	\vspace{0.3cm}
	
	We will begin by showing that, if there exists an automorphism $\alpha_\phi$ lifting $\phi$, then $\phi$ preserves the characteristic class. Let 
	$\AH\left(M\right)^\varphi$ denote the deformation quantization given by twisting the product $\star$ of $\AH\left(M\right)$ by $\varphi$, i.e. 
	\[f\star_\varphi g=\varphi_*\left(\varphi^*\left(f\right)\star\varphi^*\left(g\right)\right),\]
	where $\phi^*$ denotes the pull-back by $\phi$ and $\phi_*$ denotes the pull-back by $\phi^{-1}$. 
	We note that $\varphi_*$ is an algebra isomorphism from $\AH\left(M\right)$ to $\AH\left(M\right)^\varphi$. Now suppose there exists an extension $\alpha_\varphi$ of $\varphi$. Then we can 
	consider the composition $G_\varphi:=\varphi_*\circ \alpha_\varphi$. By a partition of unity argument (using the $\star$-product), we find that $\alpha_\varphi$ restricts to a map 
	$\AH\left(U\right)\rightarrow \AH\left(\varphi^{-1}\left(U\right)\right)$ for every open
	$U\subset M$. This shows that $G_\varphi$ must in fact be a gauge equivalence and thus the characteristic class of $\AH\left(M\right)^\varphi$ must be the same as the class of 
	$\AH\left(M\right)$ by proposition \ref{Fedisall} and theorem \ref{classification}. 
	On the other hand, it is not hard to see that $\star_\varphi$ is induced by the Fedosov connection $(\varphi^{-1})^*\nabla_A$ and thus we find that preserving $[\theta]$ is a necessary condition. 
	
	\vspace{0.3cm} 
	
	Suppose now that $\phi$ preserves the characteristic class of $\AH(M)$. Note that the pull-back $\phi^*$ is a well-defined algebra automorphism of $\Omega^\bullet\left(M;\W\right)$. Clearly, it is an isomorphism \[\phi^*\colon \Ker\nabla_A\stackrel{\sim}{\longrightarrow} \Ker\phi^*\nabla_A.\] We will show sufficiency, i.e. we will show that there exists a lift $\alpha_\phi$ of $\phi$, by constructing a map in the opposite direction given by a section 
	$U_\phi$ of $\cG_1$. This means that the composite map $U_\phi\circ \phi^*$ extends $\phi$. The construction of $U_\phi$ can also be found in \cite{N-T99} or section 5.5 of \cite{Fb}. The construction 
	of $U_\phi$ is by an iterative procedure. Note that, since $\phi$ preserves the characteristic class, we find that $\phi^*\nabla_A$ defines the same class. Moreover, since adding any central form to $\phi^*\nabla_A$ does not change its kernel, we may as well assume the curvatures (in the sense of notation \ref{curv}),  of the 
	connections $\nabla_A$ and $\phi^*\nabla_A$, are equal. 
	
	\vspace{0.3cm}

	We have fixed a lift of $\nabla_A$ to $\tg$ relative to $\nabla_0$ above, namely $A=\tilde{\sigma}(\nabla_A-\nabla_0)$, which equals $A_{-1}+A_{\N}$. 
	This provides the one-forms 
	\[r_k:=A_k-\phi^*A_k\hspace{0.2cm}\mbox{for}\hspace{0.2cm}k>0\hspace{0.5cm}\mbox{and}\hspace{0.3cm} r_0=\tilde{\sigma}\left(\nabla_0-\phi^*\nabla_0\right).\]
	Note that $r_k$ has values in $\tg_k\subset\tg$. Note also that the form $A_{-1}$ is a diffeomorphism 
	invariant (the $\omega$ in the definition cancels against the $\umega$ of the commutator in $\W$). Now, since the curvatures of $\nabla_A$ and $\phi^*\nabla_A$ are equal, we find for degree reasons that 
	\[P_\g(A_{-1})(r_0)=0.\] Now we note that $P_g(A_{-1})$ is an acyclic differential on $\Omega^\bullet(M;\W)$, which means that we may choose 
	$\delta_1\in \Omega^0\left(M,\tg_1\right)$ such that \[r_0=P_\g(A_{-1})(\delta_1).\] 
	Then, by proposition \ref{gG}, we can consider the automorphism $\Ad\exp(\delta_1):=\exp(P_g(\delta_1))$ of $\Gamma(\W)$, note that it is a section of $\cG_1$. We find that 
	\[\nabla_A-\Ad\exp(\delta_1)\circ\phi^*\nabla_A\circ\Ad\exp(-\delta_1)=0\hspace{0.1cm}\mbox{mod}\hspace{0.1cm} \g_{>0}.\]
	In other words if we replace $\phi^*\nabla_A$ by $\Ad\exp(\delta_1)\circ\phi^*\nabla_A\circ\Ad\exp(-\delta_1)$ we find that 
	$r_0=0$. This means for degree reasons that \[P_\g(A_{-1}(r_1)=0,\] so we know there exists $\delta_2$ such that 
	\[r_1=P_g(A_{-1})(\delta_2).\]
	Continuing like this we can iteratively define $\delta_i$ for all $i\in \N$. Since the 
	degrees 
	of the $\delta_i$ necessarily increase we find the well-defined section 
	\[U_\phi:=\Ad\ldots\exp\delta_i\exp\delta_{i-1}\ldots\exp\delta_2\exp\delta_1\] of $\cG_1$ such that 
	\begin{equation}\label{Fedcovariance}U_\phi\circ\phi^*\nabla_A\circ U_\phi^{-1}=\nabla_A.\end{equation} Note that this implies that 
	\[U_\phi\colon \Ker\phi^*\nabla_A\longrightarrow \Ker\nabla_A\] is a well-defined isomorphism. 
\end{proof}

\begin{remark}\label{groupac}
	Note that the proof of sufficiency in proposition \ref{lifts} provides a (highly non-unique) choice of section of the map from 
	algebra automorphisms to symplectomorphisms preserving the characteristic class $[\theta]$. However, it is not guaranteed  that such a section 
	is a group homomorphism. So let us consider, instead of a single symplectomorphism, a group $\Gamma$ of symplectomorphisms. In other words, suppose we have 
	a group $\Gamma$ acting (from the left) on $M$ by symplectomorphisms, i.e. a group homomorphism $\Gamma\rightarrow\Symp(M,\omega)$. Note that we have a map 
	\begin{equation}\label{groupaceq}
	\Hom\left(\Gamma^{op},\Aut\left(\AH(M)\right)\right)\longrightarrow \Hom\left(\Gamma,\Symp(M,\omega)_{[\theta]}\right),
	\end{equation}
	by proposition \ref{lifts} although we cannot be assured of surjectivity any longer.
\end{remark}
\begin{remark}\label{expmodcen} 
	Note that, by proposition \ref{gG}, we find that 
	$F_1\g=P_{\g}\left(\tg_{\geq1}\right)=P_\g\left(\frac{\W_{\geq3}}{i\hbar}\right)$ and $G_1=\exp F_1\g$. So we find that sections of $\cG_1$ can always locally be given as (conjugation by) exponentials of sections of $\frac{1}{i\hbar}\W_{\geq3}$. 
\end{remark}

Suppose $\alpha$ is an automorphism of $\AH(M)$, then it induces an automorphism of $\Gamma(\W)$ as follows. First we note that it induces isomorphisms
\[\hat{\alpha}_m\colon \widehat{A_{\phi(m)}(M)}\longrightarrow\widehat{A_m(M)}\] where 
$\phi$ is the symplectomorphism induced by $\alpha$. This induces in turn a diffeomorphism
$\hat{\alpha}\colon\tM\rightarrow \tM$ by pre-composition, which restricts to a diffeomorphism $\hat{\alpha}\colon \tM_r\rightarrow\tM_r$. Note that $\hat{\alpha}$ is 
$G_r$-equivariant. Thus we obtain an isomorphism $\hat{\alpha}\colon \tM_r\times_{G_r}\WW\rightarrow \tM\times_{G_r}\WW$.  Now we choose an $\Sp(2d,\R)$-equivariant section 
$F\colon\Sp_M\rightarrow\tM_r$. This means we obtain the trivialization 
$\tM_r\simeq \Sp_M\times G_1$ of the $G_1$-principal bundle $\tM_r\rightarrow \Sp_M$ and we have an isomorphism $\tM_r\times_{G_r}\WW\rightarrow \Sp_M\times_{\Sp(2d,\R)}\WW$. Thus we obtain the induced map $\hat{\alpha}\colon \W\rightarrow \W$. By definition of $\W=\Sp_M\times_{\Sp(2d)}\WW$. The following lemma shows that the ``form" of the lift obtained in the proof of \ref{lifts} is the only possibility.

\begin{lemma}\label{autoform}
	Suppose $\alpha_\phi$ is an automorphism of $\AH(M)$ lifting the symplectomorphism $\phi$, then 
	\[\widehat{\alpha_\phi}=c_\phi\circ\phi^*\] 
	for some $c_\phi\in\cG_1(M)$. 
\end{lemma}

The lemma says that the induced morphism on $\W$ is always a composition of a global part given by $\phi$ and a local (nearly inner) part $c_\phi$. In fact the decomposition of $\widehat{\alpha_\phi}$ mirrors the decomposition of $G$ given in \ref{gG}. 

\begin{proof}
	\leavevmode
	
	\noindent 
	
	We should prove that $C_\phi=\widehat{\alpha_\phi}\circ\phi_*$ defines a section of $\cG_1(M)$. This follows easily from the fact that $G/G_1\simeq\Sp(2d,\R)$ and $c_\phi$ covers 
	the identity $M\rightarrow M$.
\end{proof} 

In order to apply the Fedosov construction in the context of group actions we need to ask for a certain compatibility between the group action and a Fedosov connection. 
Combining this with the previous proposition motivates the following definition. 

\begin{definition}\label{extac}
	An extension of the left action of $\Gamma$ on $M$ by symplectomorphisms is defined as a right action $\alpha\colon \Gamma^{op}\rightarrow \Aut\left(\AH\left(M\right)\right)$ such 
	that  
	
	\begin{equation}\label{extendaction}\alpha_\gamma=c_\gamma\circ\gamma^*,\end{equation}
	for $\gamma\in\Gamma$ and where $c_\gamma\in\cG_1(M)$, such that 
	\[c_\gamma\circ\gamma^*\nabla_A\circ c_\gamma^{-1}=\nabla_A\] for some Fedosov connection $\nabla_A$ such that $\Ker\nabla_A\simeq \AH(M)$. 
	Here we have denoted the symplectomorphism corresponding to $\gamma\in\Gamma$ by $\gamma$ also. 
	
\end{definition}

Note that the definition makes sense since we can identify $\AH(M)\simeq\Ker\nabla_A$ and the 
condition on $c_\gamma$ ensures that $\alpha_\gamma(\Ker\nabla_A)=\Ker\nabla_A$. 

\begin{remark}\label{why}
	In \cite{Ginvariant,RW} and other sources the authors consider only those extensions where the $c_\gamma$ are trivial. These are of course the most natural 
	extensions (if they exist). The reason to consider not only the most natural extensions of the action to the deformation quantization is 
	that two equivalent (Fedosov) star products that allow such a natural action are not necessarily equivariantly equivalent. This means that, when one transports the 
	natural group action from one  to the other (using the equivalence) it will not be of 
	this most natural form. However, there is  in general no clear reason to prefer one (Fedosov) star product over the other, when one starts from a given characteristic 
	class. Thus, when one 
	considers all extensions, as defined in definition \ref{extac}, one becomes free to consider an arbitrary Fedosov connection with curvature in the characteristic class. 
	Another, more obvious, reason is that there may be cases where extensions as in definition \ref{extac} exist while the extension with trivial $c_\gamma$ does not. 
\end{remark}

\begin{remark}\label{comp?}
	Note that the condition of compatibility with the Fedosov connection will be superfluous in many cases, see \cite{Fp}. In particular given a symplectomorphism 
	that preserves not only the class $[\theta]$, but also a representative $\theta$ of that class we can always lift it in a way compatible with a Fedosov connection. 
\end{remark}

\section{Existence of Extensions of Group Actions}\label{4} 
In this section we investigate the problem of lifting a group action $\Gamma\rightarrow\Symp\left(M,\omega\right)_{[\theta]}$ to a group action 
$\Gamma^{op}\rightarrow\Aut\left(\AH\left(M\right)\right)$. The question of existence of lifts is exactly the question of surjectivity of the 
map \eqref{groupaceq}. We will rephrase the question in a rather compact form. We will also give some sufficient conditions for the existence of a lift and show 
that they are not necessary in general. 

\vspace{0.3cm} 

So, let us fix the group action $\Gamma\rightarrow\Symp\left(M,\omega\right)_{[\theta]}$. 
As mentioned above, we can, as in proposition \ref{lifts}, always find a map 
$\alpha\colon\Gamma^{op}\rightarrow \Aut\left(\AH\left(M\right)\right)$ such that $\alpha_\gamma\mapsto \gamma\in\Symp\left(M,\omega\right)_{[\theta]}$ 
by the map mentioned in the previous section. Note also that the proof of proposition \ref{lifts} shows that 
$\alpha_\gamma:= c_\gamma\circ\gamma^*$, where $ c_\gamma\circ \gamma^*\nabla_A\circ c^{-1}_\gamma=\nabla_A$, for some section $c_\gamma$ of $\cG_1$ and 
a given Fedosov connection $\nabla_A$. The following observation is also present in Fedosov's paper \cite{Fp}.
We have 
\[\alpha_\gamma\circ \alpha_\mu=c_\gamma\gamma\left(c_\mu\right)\circ\left(\mu\gamma\right)^*,\]
where the action of $\Gamma$ on sections of $\cG_1$ is by conjugation, i.e. $\gamma(c)=\gamma^*\circ c \circ \gamma_*$. On the other hand 
$\alpha_{\mu\gamma}=c_{\mu\gamma}\circ\left(\mu\gamma\right)^*$. So we find that the $c_\gamma$ should satisfy a cocycle condition 
in order for $\alpha$ to be a group homomorphism. Indeed, $\alpha$ defines a group action exactly when 
\begin{equation}\label{cocyclecond}c_\gamma\gamma\left(c_\mu\right)c_{\mu\gamma}^{-1}=\Id\end{equation} for all $\gamma,\mu\in \Gamma$.

\begin{corollary}\label{cocycllllllllll} 
	The action $\Gamma\rightarrow \Symp\left(M,\omega\right)_{[\theta]}$ lifts to an action on the deformation quantization in the sense of definition \ref{extac} 
	iff there exist sections $c_\gamma$ of $\cG_1$, such that $ c_\gamma\circ \gamma^*\nabla_A\circ c^{-1}_\gamma=\nabla_A$
	for all $\gamma\in \Gamma$ and some Fedosov connection $\nabla_A$ (such that $\Ker\nabla_A\simeq \AH(M)$), that form a cocycle.  
\end{corollary}

\begin{remark}
	Although it is in general not easy to check the cocycle condition \eqref{cocyclecond}, more can be said about the question of existence in more general terms. As noted in
	\cite{Fb}, one can construct a Fedosov connection $\nabla_A$ with characteristic class $[\theta]$
	\emph{uniquely}, given a representative $\theta\in\Omega^2(M)\hhbar$ and a symplectic torsion-free connection. Moreover, given an invariant linear 
	connection, one can construct an invariant 
	torsion-free symplectic connection in a canonical way, this is done (again) by the ``Hesse trick" also applied in \cite{PPT2}.   
\end{remark} 

\begin{proposition}\label{trivialextensions}
	Suppose there exists an invariant linear connection $\nabla_{00}$ on $M$, i.e. such that $\gamma^*\nabla_{00}=\nabla_{00}$ for all $\gamma\in\Gamma$. Then the action extends to any 
	deformation quantization which has characteristic class in the image of the map 
	\begin{equation}\label{equivmap}\frac{\omega}{i\hbar}+\CH^2_\Gamma\left(M\right)\hhbar\longrightarrow \frac{\omega}{i\hbar}+\CH^2\left(M\right)^\Gamma\hhbar,\end{equation}
	where we denote the cohomology of the complex $\left(\Omega^\bullet\left(M\right)^{\Gamma}\hhbar,d\right)$ by $\CH^\bullet_\Gamma\left(M\right)\hhbar$ and 
	we denote the invariants by a superscript $\Gamma$. 
\end{proposition}
\begin{proof} 
	
	\leavevmode
	
	\noindent The proposition follows immediately from the remark above and is also noted in \cite{Fb}. For the situation as described in the lemma it is possible to construct a Fedosov 
	connection $\nabla_A$ 
	such that $\gamma^*\nabla_A=\nabla_A$ for all $\gamma\in\Gamma$. Then $\gamma^*$ defines an automorphism of $\Ker\nabla_A$ and thus $\alpha_\gamma=\gamma^*$ extends the action. 
\end{proof}

\begin{remark}
	Note that the class $\frac{\omega}{i\hbar}+[\omega]$ (and similar classes) will always be in the image of the map stated above. Note also that proposition
	\ref{trivialextensions} is comparable to the results in \cite{Ginvariant}, there the authors consider more restrictive extensions (i.e. $\alpha_\gamma=\gamma^*$) and 
	show that any class in the left hand cohomology induces an \emph{invariant} star product up to \emph{equivariant} equivalence. These results where extended to 
	include quantum moment maps fairly recently in \cite{RW}. 
\end{remark}

\begin{corollary}\label{finiteistrivial}
	Suppose $\Gamma$ is a compact Lie group,  
	then any (smooth) action of $\Gamma$ by symplectomorphisms can be extended to any deformation quantization by $\alpha_\gamma=\gamma^*$. 
\end{corollary}
\begin{proof} 
	
	\leavevmode 
	
	\noindent By averaging an arbitrary linear connection one obtains an invariant linear connection and by averaging an arbitrary representative of a characteristic class one obtains a
	representative of 
	the pre-image in $\frac{\omega}{i\hbar}+\CH^2_\Gamma\left(M\right)\hhbar$. 
\end{proof}

\begin{remark}\label{extra}
	Note that the proof of corollary \ref{finiteistrivial} actually also applies to the case where $\Gamma$ is not itself a compact Lie group, but the action of $\Gamma$ factors through the (smooth) action of a compact Lie group.
\end{remark}

\begin{remark} 
	It is possible to phrase the first condition in proposition \ref{trivialextensions} in terms of a cohomological obstruction in the following way. Consider the right-action of $\Gamma$ on 
	$\Omega^1\left(M,\End\left(TM\right)\right)$ by pull-back. Given an affine connection $\nabla_{00}$, we define 
	\[\hspace{0.3cm}D\colon \Gamma \longrightarrow \Omega^1\left(M,\End\left(TM\right)\right)\] 
	\[\gamma\longmapsto \gamma^*\nabla_{00}-\nabla_{00}\]
	Note that this defines a group cocycle in the group cohomology complex of $\Gamma^{op}$ with values in the vector space 
	$\Omega^1\left(M,\End\left(TM\right)\right)$. Clearly, if $\nabla_{00}$ is 
	invariant we find that $D=0$. It is easy
	to check that the class of $D$ in $\CH^1\left(\Gamma^{op},\Omega^1\left(M,\End\left(TM\right)\right)\right)$ does not depend on $\nabla_{00}$ and in fact $[D]=0$ if and only if
	there exists an invariant affine connection 
	on $M$. The class defined above can be identified as a certain notion of the Atiyah class \cite{AC2}.
\end{remark}

Proposition \ref{trivialextensions} gives a sufficient condition for the existence of a lift of a group action. It includes many group actions on manifolds, 
so a natural question is whether it is possible to extend actions that do not satisfy the hypotheses 
of the proposition. In other words, one might wonder if there exist actions that allow extension, but do not satisfy the hypotheses of proposition 
\ref{trivialextensions}, i.e. one wonders whether the conditions of proposition \ref{trivialextensions} are necessary.

\begin{example}\label{worthexample}
	Consider the cotangent bundle $\left(T^*S^2, d\eta\right)$ of the $2$-sphere with the canonical symplectic structure $d\eta$. Let $\Gamma$ be the group of orientation preserving 
	diffeomorphisms of $S^2$ that fix the equator $Eq\colon S^1\hookrightarrow S^2$ pointwise. 
	It is an elementary fact of symplectic geometry, found for instance in \cite{ACDS}, that the group of diffeomorphisms of a manifold lifts to symplectomorphisms of the corresponding 
	cotangent bundle. Note that \[\CH^2\left(T^*S^2\right)\hhbar=\C\hhbar\pi^*[\omega],\] where $\pi$ denotes the projection $T^*S^2\rightarrow S^2$ and $\omega$ denotes the standard 
	symplectic structure of $S^2$. Since $\Gamma$ consists of orientation preserving diffeomorphisms we find that the class $\pi^*[\omega]\in \CH^2\left(T^*S^2\right)^\Gamma$ is invariant. 
	On the other hand, if $\beta$ is any $1$-form 
	on $T^*S^2$ such that 
	$\pi^*\omega+d\beta\in\Omega^2\left(T^*S^2\right)^\Gamma$, then we find 
	\[\omega+d\left(z^*\beta\right)=z^*f^*_\sharp\left(\pi^*\omega+d\beta\right)= f^*z^*\left(\pi^*\omega+d\beta\right)=f^*\left(\omega+dz^*\beta\right)\] for all $f\in\Gamma$, here we have 
	denoted the zero-section of 
	$T^*S^2\rightarrow S^2$ by $z$ and the symplectomorphism of $T^*S^2$ induced by $f$ by $f_\sharp$. However, note that any non-zero two-form on $S^2$ will have support on some open $U$ which is disjoint 
	from $Eq(S^1)$ and subject to a multitude of ``local'' diffeomorphisms in $\Gamma$. So since $\omega+dz^*\beta$ is invariant under all such diffeomorphisms we find that it vanishes identically,
	but this implies that $[\omega]=0$. Thus we are led to a contradiction, which shows that $[\omega]$ is not in the image of the map \eqref{equivmap}, in fact this map is $0$. 
	Thus this action does not satisfy the criteria of proposition \ref{trivialextensions} and there does not exist any invariant Fedosov connection with non-trivial characteristic class for this action. 
	On the other hand, we note that, by considering the clutching construction of vector bundles on $S^2$, we can lift the action of $\Gamma$ to an action on the line bundle corresponding to the 
	inclusion $S^1\hookrightarrow \C^\times$ (which has Chern class $[\omega]$). This means we can extend the action of $\Gamma$ to differential operators on smooth sections of this line bundle. 
	Thus by the results in \cite{repline} we find that the action of $\Gamma$ does lift to the deformation quantization with characteristic class $\hbar[\pi^*\omega]$.

\end{example}

By example \ref{worthexample}, we see that although the conditions of proposition \ref{trivialextensions} are sufficient to conclude existence of extended group actions (as defined in definition \ref{extac}), they 
are not necessary. 

	\vspace{0.3cm}

It is notable that by the results in \cite{repline} the question of lifting an action to deformation quantization reduces, in the case of deformations of cotangent bundles with 
characteristic classes as in Example \ref{worthexample}, to lifting the action from the manifold to a line bundle with the corresponding Chern class. This last problem essentially comes down to 
showing that the Chern class extends to an equivariant class. Thus the condition \eqref{cocyclecond} seems to imply there is a deformation quantization analog of equivariant cohomology for which the 
extension of the characteristic class to an equivariant class is a necessary and sufficient condition for the action to extend to the corresponding deformation.

\section{Classification of Extended Group Actions}

In this section we will turn to the question of classification of extended group actions. 
We will show that extended group actions are classified, up to a technical condition, by the first cohomology of the group $\Gamma$ with values in a certain non-Abelian group $\GZ$. We will finish the classification, by first providing the essential tools for computing this first cohomology and subsequently considering some examples. The classification will be carried out relative to a given extended group action. So for this section we will fix a Fedosov connection $\nabla$ (dropping the subscript $A$) and 
an extended action $\alpha\colon \Gamma^{op}\rightarrow \Aut\left(\AH\left(M\right)\right)$. 

\subsection{Abstract Classification}\label{5.3.1}

In this section we will give a classification of the lifts of group actions as in definition \ref{extac} in abstract terms, up to a certain technical condition. We will provide methods of computation of 
these abstract objects in the next section. To do this, we will first, following Fedosov \cite{Fp, Fb}, introduce an extension of the Weyl algebras bundle 
such that a subgroup of the group of invertible sections surjects locally onto the sections of $\cG_1$ in a natural way. This will allow us to understand the sheaf 
$\cG_1$ better, in order to provide tools for computation in the following section and also define a subgroup of the total sections providing the abstract classification.

\begin{definition}\label{W+}
	Consider the algebra $\WW_{(\hbar)}:=\WW_\hbar\otimes \C[\hbar^{-1},\hbar]\!]$ where the tensor product is over $\Ch$, i.e. $\WW_{(\hbar)}=\WW_\hbar[\hbar^{-1}]$. Note that $\WW_{(\hbar)}$ carries a grading induced by the grading of $\WW_\hbar$, i.e. $|\hbar^{-1}|=-2$. We define the algebra $\WW_\hbar^+\subset \WW_{(\hbar)}$ as the subalgebra $F_0\WW_{(\hbar)}$ of elements with degree greater than $0$. In other words we 
	allow power series with negative powers of $\hbar$ as long as the total degree is still greater than $0$.
	Similarly, we denote the bundle associated to $\Sp_M$ with fibers 
	given by $\WW_\hbar^+$ by $\W^+$ and the bundle associated to $\Sp_M$ with fibers given by $\WW_{(\hbar)}$ by $\W^f$.
\end{definition}

\begin{remark}\label{bundlemark}
	Note that the Fedosov connection $\nabla$ is well-defined on the bundle $\mathcal{W}^f$. We will denote the center of $\W^+$ and $\W$ by $\cZ$.
	Note that $\cZ\simeq C_M^\infty\hhbar$ by the inclusion $\Ch\hookrightarrow \WW_\hbar^+$. Here  $\W$, $\cZ$, $C^\infty_M$ and so on denote sheaves (it should be obvious which sheaves). 
\end{remark}

\begin{definition}\label{T_F} 
	We define the sheaf of \emph{fiberwise transformations} by assigning the sections of $\mathcal{W}^+(U)$ which are given by 
	exponentials of elements of $\left(\frac{1}{i\hbar}\mathcal{W}(U)\right)_{\geq1}$ to the open subset $U$. We will denote this sheaf by $\TF$.
	Note that $\TF$ is a sheaf of groups by the Campbell-Baker-Hausdorff formula.
\end{definition}

\begin{remark}\label{reltoG} 
	Note that there is map \[\Ad\colon\TF\longrightarrow \cG_1\] given by assigning to the section $E$ of $\TF$ the automorphism of 
	conjugation by $E$. Note that, by the proposition \ref{gG}, we find that $\Ad$ is locally surjective. 
\end{remark}

Suppose that, for $E_\phi, E'_\phi\in \TF$, $U_\phi:=\Ad E_\phi$ and $U_\phi':=\Ad E'_\phi$ both satisfy \eqref{Fedcovariance}. Then we find that, 
denoting $E=E'_\phi E^{-1}_\phi$, \begin{equation}\label{fedinvar}\Ad E\circ\nabla\circ\Ad E^{-1}=\nabla.\end{equation} Conversely, if 
$U_\phi$ satisfies \eqref{Fedcovariance} and $E$ satisfies 
\eqref{fedinvar}, then clearly $U_\phi':=\Ad E\circ U_\phi$ also satisfies \eqref{Fedcovariance}. 

\begin{remark}\label{tech} 
	The discussion above suggests the following technical condition on the kind of actions we allow. Namely, we will only consider those actions that are of the form 
	\[\gamma\mapsto \Ad E_\gamma\circ\alpha_\gamma,\]
	with $E_\gamma\in\TF$ for all $\gamma\in\Gamma$. 
\end{remark}

\begin{remark}\label{why?}
	It serves now to compare the technical conditions in definition \ref{extac} and remark \ref{tech}. First of all we note that they are in fact compatible. 
	Secondly we compare them to the notion of isomorphism (and automorphism) of quantum algebras as defined in \cite{Fb} section 5.5. We note that the condition 
	in definition \ref{extac} is in fact weaker than the condition of a connection preserving isomorphism in \cite{Fb}. On the other hand, the condition in remark 
	\ref{tech} is, in the terms of \cite{Fb}, exactly the condition that the action is given, relative to $\alpha_\gamma$, by a fiberwise automorphism preserving the connection $\nabla$. Thus, if we use the definition of automorphism in section 5.5 of \cite{Fb}, the only technical condition is preservation of the Fedosov 
	connection $\nabla$. The reason, also offered in \cite{Fb}, to consider these kinds of actions seriously is that they correspond to the time
	evolution operators through Heisenberg's equations of motion.
	\end{remark}

\begin{definition} 
	We define \emph{Fedosov's fiberwise $\nabla$ preserving isomorphisms} by  
	\[\G:=\left\{E\in\TF\mid\nabla\left(E^{-1}\right)E\in\Omega^1\left(M\right)\hhbar\right\}.\] 
	
\end{definition}
We should mention that, although it is not presented quite in this form, a lot of the following (excluding the classification) is implicit in the work of Fedosov \cite{Fb, Fp}. We should justify the nomenclature of $\G$.

\begin{lemma} 
	The section $E\in \TF$ satisfies \eqref{fedinvar} if and only if $E\in\G$. 
\end{lemma}
\begin{proof} 
	
	\leavevmode
	\noindent  Suppose $E\in\TF$ satisfies 
	\eqref{fedinvar}, then for all $\sigma\in\W$ we find 
	\[\nabla\sigma=E^{-1}\left(\nabla\left(E\sigma E^{-1}\right)\right)E=\nabla\sigma+[E^{-1}\nabla E,\sigma],\] which implies $E\in\G$, since 
	$E^{-1}\nabla E=-\nabla(E^{-1})E$. The equation above also shows the converse statement. 
\end{proof}

\begin{notation} 
	We denote the group of automorphisms of $\W$ that are given by sections of $\TF$ and induce automorphisms of $\Ker \nabla$ by $\Loc\left(\W\mid\nabla\right)$. 
\end{notation}

\begin{lemma}\label{Gstuff}
	\leavevmode
	\begin{description}
		\setlength{\itemsep}{.01mm}
		\item[(i)]  $\G$ is a subgroup of the invertibles $\left(\W^+\right)^\times$ of $\W^+$. 
		\item[(ii)] $\mathcal{Z}^\times\triangleleft\G$ and $\left(\Ker\nabla\right)^\times\triangleleft\G$. 
		\item[(iii)] $\bigslant{\G}{\mathcal{Z}^\times}\simeq\Loc\left(\W\mid\nabla\right)$.
		\item[(iv)] $\G$ forms a sheaf of groups on $M$.
		\item[(v)]  For any open $V\subset M$ such that $\CH^1\left(V\right)=0$ we have 
		\begin{equation}\label{decompG} \G|_V=\cZ^\times|_V\cdot\AH(V)^\times\end{equation}
		\item[(vi)] $\alpha_\gamma\left(\G\right)\subset \G$ for all $\gamma\in\Gamma$. 
	\end{description}
	
\end{lemma}
\begin{proof} 
	
	\leavevmode
	
	\noindent ``(i)'' Suppose $E,B\in\G$, we should show that $E^{-1}$ and $EB$ are also in $\G$. We have $E^{-1}\in\G$, since 
	\[E^{-1}\nabla E=\Ad E\left(E^{-1}\nabla E\right)=\left(\nabla E\right)E^{-1}=-E\nabla E^{-1}.\] Similarly, we have  
	$EB\in\G$, since \[B^{-1}E^{-1}\nabla EB=\Ad B^{-1}\left(E^{-1}\nabla E\right)+B^{-1}\nabla B=E^{-1}\nabla E+B^{-1}\nabla B.\]
	``(ii)'' Since $\mathcal{Z}^\times$ is central we see that $\nabla z=dz$ for all $z\in\mathcal{Z}^\times$, this shows that $\mathcal{Z}^\times$ is a subgroup. It is 
	a normal subgroup because it is central. Of course $\nabla k=0$ for all $k\in\left(\Ker\nabla\right)^\times$ showing that it is a subgroup. It is normal since 
	\begin{equation}\label{kernormal}\nabla\left(EkE^{-1}\right)=\nabla\left(E\right)kE^{-1}+Ek\nabla E^{-1}=\Ad E\left([E^{-1}\nabla\left(E\right),k]\right)=0\end{equation} for all $k\in\left(\Ker\nabla\right)^\times$ and $E\in\G$.\\
	``(iii)'' Consider the group homomorphism \[\Ad\colon\G\longrightarrow\Loc\left(\W\mid\nabla\right)\] given by restricting the group homomorphism 
	$\Ad\colon \TF\rightarrow \cG_1$. Note that it is well-defined, since \ref{kernormal} also holds for 
	$k\in\Ker\nabla$. Suppose $\alpha=\Ad B\in\Loc\left(\W\mid\nabla\right)$, then we find that 
	\[0=\Ad B^{-1}\circ\nabla\circ\Ad B\left(k\right)=[B^{-1}\nabla\left(B\right),k]\] for all $k\in\Ker\nabla$. So $\nabla\left(B\right)B^{-1}$ is in the centralizer of $\Ker\nabla$. 
	Given $k_x\in\W_x$ at some $x\in M$ we can (by parallel
	transport) always find a section $k\in\Ker\nabla$ such that $k\left(x\right)=k_x$. This shows that the centralizer of $\Ker\nabla$ is simply the center and so $B\in\G$. 
	Thus, $\Ad$ is surjective 
	and, since $\Ker\Ad=\mathcal{Z}^\times$ (another proof of the first part of (ii)), we find that (iii) holds. \\
	``(iv)'' This should be clear from the definition of $\G$. \\
	``(v)'' Suppose $V\subset M$ such that $\CH^1\left(V\right)=0$. Then suppose $E\in\G|_V$ and denote the implied central one-form by $\beta=E^{-1}\nabla E$. Note that $\nabla E=\beta E$. We have 
	\[d\beta=\nabla\beta=\nabla\left(\nabla\left(E\right)E^{-1}\right)=\beta\wedge\beta=0.\] 
	So, since $\CH^1\left(V\right)=0$, we find that $\beta=d\alpha$ for some $\alpha\in C^\infty\left(V\right)\hhbar$. Then \[\nabla\left(e^{-\alpha}E\right)=-d\alpha e^{-\alpha}E+e^{-\alpha}\beta E=0,\] 
	so we find that 
	$E=e^{\alpha}e^{-\alpha}E$, i.e. $E$ is a product of a central section $e^\alpha$ and a flat section $e^{-\alpha}E$. Finally we note that, if 
	$a_f\in\W^+$ is flat, then it cannot contain any negative powers of $\hbar$, since it is uniquely determined by its image under 
	$\WW_\hbar^+\rightarrow \Ch$. So we see that $\Ker\nabla\simeq \AH$ still holds when we consider $\nabla$ as acting 
	on $\W^+$. \\
	``(vi)'' Suppose $\gamma\in\Gamma$ and $E\in\G$, then $\Ad \alpha_\gamma\left(E\right)=\alpha_\gamma\circ\Ad E\circ \alpha_\gamma^{-1}\in\Loc\left(\W\mid\nabla\right)$. 
	So, by (iii) and the fact 
	that the $\alpha_\gamma$ define automorphisms of $\Ker\nabla$ and of $\cZ$, we have $\alpha_\gamma\left(E\right)\in\G$.  
\end{proof}

\begin{notation}
	From now on we will be considering group cohomology. Since we are considering right actions, i.e. $\Gamma^{op}$, we should be writing the group cohomology 
	with values in $B$ as $\CH^\bullet(\Gamma^{op};B)$. For notational convenience and since it will not play any role, we will drop the superscript $op$. We will also denote the quotient $\G/\mathcal{Z}^\times$ by $\GZ$. 
\end{notation}

\begin{theorem}\label{classif}

	The group cohomology pointed set $\CH^1\left(\Gamma;\GZ\right)$ classifies the actions of 
	$\Gamma$ on $\AH(M)$ that extend a given action on $M$, in the sense of \ref{extac} and satisfying the condition in remark \ref{tech}, 
	up to conjugation by a fixed element of $\Loc\left(\W\mid\nabla\right)$.  
	
\end{theorem}
\begin{proof} 
	
	\leavevmode
	
	\noindent Note first of all that the action  $\alpha$ also descends to $\GZ$, since it restricts to an automorphism of the center (via the identification 
	$\cZ\simeq C^\infty\left(M\right)\hhbar$ this is simply the action induced by the action on the manifold). 
	Suppose $\tilde{S}\colon\Gamma\rightarrow\GZ$ is a cocycle for the action $\alpha$. Pick any lift $S$ of $\tilde{S}$ to $\G$. Then we have 
	\[S_\gamma \alpha_\gamma\left(S_\mu\right)S_{\mu\gamma}^{-1}\in\mathcal{Z}^\times\hspace{0.3cm}\forall\hspace{0.1cm}\gamma,\mu\in\Gamma.\]
	We see that \[\beta\colon \Gamma^{op}\longrightarrow\Aut\left(\AH\left(M\right)\right),\]
	defined by $\beta_\gamma=\Ad S_\gamma\circ \alpha_\gamma$, is a group homomorphism. Note that $\beta$ is well-defined by point (iii) of \ref{Gstuff} and it is a group homomorphism  since 
	\[\beta_\gamma\circ \beta_\mu=\alpha_\gamma\circ\Ad  \alpha_\gamma^{-1}\left(S_\gamma\right)S_\mu\circ \alpha_\mu=\alpha_\gamma\circ\Ad \alpha_\gamma^{-1}\left(S_{\mu\gamma}\right)
	\circ \alpha_\mu=\Ad S_{\mu\gamma}\circ \alpha_\gamma\circ \alpha_\mu=
	\beta_{\mu\gamma}.\]
	Note that $\beta$ does not depend on the particular lift of $\tilde{S}$. 
	
	\vspace{0.3cm}
	
	Conversely, suppose $\beta_\gamma=\alpha_\gamma\circ\Ad E_\gamma$ where $E\colon\Gamma\rightarrow \G$ and $\beta$ defines an action. Note that $\beta$ only depends on the 
	induced map $\tilde{E}$ into $\GZ$. Since $\beta_{\mu\gamma}=\beta_\gamma\circ \beta_\mu$, we find immediately that 
	\[\Ad \alpha_\gamma\left( E_\gamma\right) \alpha_{\mu\gamma}\left(E_\mu\right)=\Ad \alpha_{\mu\gamma}\left(E_{\mu\gamma}\right),\] which implies that 
	$\gamma\mapsto \alpha_\gamma\left(\tilde{E}_\gamma\right)$ 
	defines a cocycle. Note that the action corresponding to it by the construction above is $\beta$. 
	
	\vspace{0.3cm}
	
	Finally, let us pass to cohomology. So, suppose $\tilde{S}$ and $\tilde{S'}$ are two cohomologous $\Gamma^{op}$-cocycles in $\GZ$. Denote by $\beta$ and $\beta'$ the corresponding actions. 
	Then there is an element 
	$\tilde{C}\in\GZ$ such that 
	\[\tilde{C}\tilde{S}_\gamma \alpha_\gamma\left(\tilde{C}^{-1}\right)=\tilde{S'}_\gamma.\]
	Then, picking any lifts $C,S$ and $S'$ of $\tilde{C}$, $\tilde{S}$ and $\tilde{S'}$ to $\G$, we find the equations 
	\[\Ad CS_\gamma \alpha_\gamma\left(C^{-1}\right)=\Ad S'_\gamma.\] 
	So, 
	\[\beta'_\gamma=\Ad S'_\gamma\circ \alpha_\gamma=\Ad CS_\gamma \alpha_\gamma\left(C^{-1}\right)\circ \alpha_\gamma=\Ad C\circ \beta_\gamma\circ\Ad C^{-1}\]
	for all $\gamma\in\Gamma$. 
	Conversely, if $\beta, \beta'$ are two actions with corresponding cocycles $\tilde{S}$ and $\tilde{S}'$ and lifts $S$, $S'$ and $C\in\G$ satisfies 
	$\Ad C\circ \beta_\gamma\circ \Ad C^{-1}=\beta'_\gamma$ for all $\gamma\in\Gamma$, then we have $\Ad CS_\gamma \alpha_\gamma\left(C^{-1}\right)=\Ad S'_\gamma$ and thus 
	$\tilde{S}$ and $\tilde{S}'$ are 
	cohomologous. 
\end{proof}

\subsection{Computational Tools}\label{5.3.2} 

In this section we will provide some tools that should aid in the concrete computation of the pointed sets $\CH^1\left(\Gamma;\GZ\right)$ of equivalence classes of extended group actions. 
The main burden of proof will be in showing that there exists a surjective map $\mathbb{D}$ from $\G$ to $Z^1(M)\hhbar$ (formal power series of closed one forms). This will provide the 
commuting diagram 
\begin{equation}\label{computingsquare}
\begin{tikzpicture}[baseline=(current  bounding  box.center)]
\draw node at (0,1.5)       {$1$};
\path [draw, ->] (0,1.3)--(0,.3) ;
\draw node at (3,1.5)     {$1$};
\path [draw, ->] (3,1.3)--(3,.3) ;
\draw node at (6,1.5)       {$1$};
\path [draw, ->] (6,1.3)--(6,.5) ;
\draw node at (-2.5,0)       {$1$};
\path [draw, ->] (-2.3,0)--(-.6,0) ;
\draw node at (0,0)       {$\Cht$};
\path [draw, ->] (.5,0)--(2.7,0) ;
\path[draw, ->]   (0,-.3) -- (0,-1.7);
\draw node at (3,0)     {$\cZ^\times$};
\path [draw, ->] (3.4,0)--(4.9,0) ;
\path[draw, ->]   (3,-.3) -- (3,-1.7);
\draw node at (6,0)       {$\bigslant{\mathcal{Z}^\times}{\Cht}$};
\path [draw, ->] (7,0)--(8.3,0) ;
\path[draw, ->]   (6,-.4) -- (6,-1.7);
\draw node at (8.5,0)       {$1$};
\draw node at (-2.5,-2)       {$1$};
\path [draw, ->] (-2.3,-2)--(-.6,-2) ;
\draw node at (0,-2)    {$\AH^\times$};
\path [draw, ->] (.5,-2)--(2.7,-2) ;
\path[draw, ->]   (0,-2.3) -- (0,-3.5);
\draw node at (3,-2)  {$\G$};
\path [draw, ->] (3.4,-2)--(5.1,-2) ;
\path[draw, ->]   (3,-2.3) -- (3,-3.7);
\draw node at (6,-2)    {$Z^1(M)\hhbar$};
\path [draw, ->] (7,-2)--(8.3,-2) ;
\path[draw, ->]   (6,-2.3) -- (6,-3.7);
\draw node at (8.5,-2)       {$0$};
\draw node at (-2.5,-4)       {$1$};
\path[draw, ->]   (-2.3,-4) -- (-1.2,-4);
\draw node at (0,-4)      {$\bigslant{\AH^\times}{\Cht}$};
\path[draw, ->]   (1,-4) -- (2.7,-4);
\path[draw, ->]   (0,-4.4) -- (0,-5.3);
\draw node at (3,-4)    {$\GZ$};
\path[draw, ->]   (3.3,-4) -- (5.3,-4);
\path[draw, ->]   (3,-4.4) -- (3,-5.3);
\draw node at (6,-4)      {$T^1_\hbar(M)$};
\path[draw, ->]   (6.7,-4) -- (8.3,-4);
\path[draw, ->]   (6,-4.4) -- (6,-5.3);
\draw node at (8.5,-4)       {$0$};
\draw node at (0,-5.5)       {$1$};
\draw node at (3,-5.5)     {$1$};
\draw node at (6,-5.5)       {$0$};
\draw node at (8.7,-5.5)       {,};
\end{tikzpicture}
\end{equation}
\noindent which has exact columns and exact rows and where \[T^1_\hbar(M):=\bigslant{\CH^1(M;\C)}{\CH^1(M;\Z)}\oplus\hbar \CH^1(M)\hhbar.\]
This matches well with the appearance of $\hbar T^1_\hbar(M)$ as a parametrization of equivalence classes
of certain formal representations and equivalence classes of certain formal connections in \cite{repline}. Note that, although we will be working with non-Abelian cohomology, we still obtain 
(truncated) exact sequences from short exact sequences of coefficient groups (see section 2.7 of \cite{NAC}). In the following section we will show how one can exploit this diagram to compute 
$\CH^1(\Gamma;\GZ)$.

\begin{notation} 
	From now on we will identify $\Ker\nabla$ with the deformation quantization $\AH$. We will often implicitly identify $\mathcal{Z}\simeq C^\infty\left(M\right)\hhbar$. For 
	elements of 
	graded algebras (or invertibles of graded algebras) a subscript will always refer to the degree in this section.
\end{notation}
\begin{lemma} 
	We have the following short exact sequence of sheaves of groups 
	\begin{equation}\label{sessheaves}\begin{matrix} 1&\rightarrow&{\AH}^\times&\hookrightarrow&\G&\displaystyle\stackrel{\D}{\longrightarrow}&Z^1\hhbar&\rightarrow&0\\
	&           &          &               &          g       &\longmapsto    & g^{-1}\nabla g        &           &  .
	\end{matrix}\end{equation}
	
\end{lemma}
\begin{proof}
	
	\leavevmode
	
	\noindent Note first that $\D$ is well-defined since $d\left(g^{-1}\nabla g\right)=\nabla\left(g^{-1}\nabla g\right)=-\left(g^{-1}\nabla g\right)^2=0$. The proof follows from the decomposition 
	\eqref{decompG}
	and the fact that de Rham cohomology vanishes locally.
\end{proof}
We will show that the above sequence induces an exact sequence of groups on global sections. Note that this means we should prove surjectivity of $\D$. To do this we 
will use \v{C}ech 
cohomology 
and some facts about rings of formal power series. So let us fix a good cover $\mathcal{U}=\{U_i\}_{i\in J}$ and recall that for smooth manifolds this choice will not play a significant role. Recall 
also that $\AH\left(M\right)$ and $C^\infty\left(M\right)\hhbar$ are complete with respect to the $\hbar$-adic topology. 

\begin{definition}\label{hadic}
	The $\hbar$-adic topology is given by the norm $\norm{f}=2^{-k}$, where $k$ is the smallest 
	non-negative integer such that $f_k\neq0$, it satisfies the (in)equalities $\norm{f+g}\leq\max\{\norm{f},\norm{g}\}$ and $\norm{fg}=\norm{f}\norm{g}$.
\end{definition}
We have the following fact (see chapter 3 of \cite{AC}). 
\begin{lemma}\label{expLog}
	
	Suppose $R$ is a commutative unital ring that contains a copy of the rationals, then $\left(\hbar R\hhbar,+\right)\simeq\left(1+\hbar R\hhbar,\cdot\right)$ by the maps 
	\[ \mbox{exp}\left(f\right)=\displaystyle\sum_{n=0}^\infty \frac{f^n}{n!} \hspace{0.3cm}\mbox{and}\hspace{0.3cm} 
	\mbox{Log}\left(1+f\right)=\displaystyle\sum_{n=1}^\infty \frac{\left(-f\right)^n}{n}.\]
\end{lemma}

\begin{lemma}\label{f01exact}
	
	Suppose $f\in\left(C^\infty\left(M\right)\hhbar\right)^\times\simeq\mathcal{Z}^\times$ and $f_0=1$, then $\frac{df}{f}$ is exact.
\end{lemma}

\begin{proof}
	
	\leavevmode
	
	\noindent It is easily verified that $d$ satisfies the product rule on formal power series. This shows that we have $d\exp{g}=\left(\exp{g}\right)dg$ for all $g\in\hbar C^\infty\left(M\right)\hhbar$. 
 Now let $g=\mbox{Log}\left(f\right)$ and we see that $\frac{df}{f}=dg$. 
\end{proof}

\begin{proposition}\label{thing1} 
	
	Let us denote the subgroup of exact forms in $Z^1\left(M\right)\hhbar$ by $d\mathcal{Z}$ and the restriction of $\mathbb{D}$ to $\mathcal{Z}^\times$ by $D$. Then we have 
	$d\mathcal{Z}\subset\Ima D$ and $\Ima D\diagup d\mathcal{Z}\simeq\CH^1\left(M;\Z\right)$. 
	
\end{proposition}

\begin{proof}
	
	\leavevmode
	
	\noindent Suppose $dg\in d\mathcal{Z}$ such that $\left(dg\right)_0=0$. Then we may assume that $g\in\hbar C^\infty\left(M\right)\hhbar$. We want to show that there exists an $f\in\mathcal{Z}^\times$ 
	such that $Df=dg$. Clearly $f=\exp\left(g\right)$ will do the
	job. This shows that every exact form $dg$ with 
	$\left(dg\right)_0=0$ is in $\Ima D$. Now suppose $dg$ is a general element of $d\mathcal{Z}$, then $\exists f\in\mathcal{Z}^\times$ such that  
	$Df=\hbar \left(dg\right)_1+\hbar^2 \left(dg\right)_2+\ldots$ , but then $dg=De^{g_0}+Df=D\left(e^{g_0}f\right)$. This shows 
	the first claim of the proposition.
	
	\vspace{0.3cm}
	
	To show the second claim consider the map \[C_\mathcal{U}\colon\Ima D\longrightarrow \chec^1\left(\mathcal{U}; \C\right)\hhbar\] where 
	$C_\mathcal{U}\left(Df\right)$ is represented by $\eta\left(i,j\right)=g\left(i\right)-g\left(j\right)$ if $Df|_{U_i}=dg\left(i\right)$. We leave the routine check that this map 
	is a well-defined group homomorphism to the reader. 
	Note that we have $\Ker C_\mathcal{U}=d\mathcal{Z}$ and in fact the map $C_\mathcal{U}$ is simply given by the usual map which implements the isomorphism of de Rham en \v{C}ech cohomology.  
	It is left to 
	show that $\Ima C_\mathcal{U}\simeq \CH^1\left(M;\Z\right)$. Note that every element in $\mathcal{Z}^\times$ can be written as a product of a nowhere vanishing function and  a 
	function as in lemma \ref{f01exact}.
	So we find that $C_\mathcal{U}\left(Df\right)=C_\mathcal{U}\left(Df_0\right)=[\eta]$, where we may set $\eta\left(i,j\right)=g\left(i\right)-g\left(j\right)$ with 
	$e^{g\left(i\right)}=f_0|_{U_i}$. So we see that 
	$[\eta]\in\chec^1\left(\mathcal{U};\Z\right)\hookrightarrow\chec^1\left(\mathcal{U};\C\right)\hhbar$. Where the inclusion comes from the exact sequence of sheaves 
	\begin{equation}\label{expconstant}0\rightarrow\Z\stackrel{2\pi i\cdot}{\longrightarrow}\C\stackrel{e^{\cdot}}{\longrightarrow}\C^\times\rightarrow 1,\end{equation}
	since it is also a short exact sequence of groups. So we find the inclusion $\Ima C_\mathcal{U}\subset \chec^1\left(\mathcal{U},\Z\right)$. On the other hand consider the exact sequence of sheaves 
	\begin{equation}\label{expsmooth}0\rightarrow\Z\stackrel{2\pi i\cdot}{\longrightarrow}C^\infty\left(M\right)\stackrel{e^{\cdot}}{\longrightarrow}C^\infty\left(M\right)^\times\rightarrow 1.\end{equation}
	Now we see that the first connecting map $\partial$ in the corresponding long exact sequence in \v{C}ech cohomology is surjective, since 
	$\chec^1\left(\mathcal{U};C^\infty\left(M\right)\right)=0$, and $\partial\left(f_0\right)=[\lambda]$ where 
	$\lambda$ is given by $\lambda\left(i,j\right)=g\left(i\right)-g\left(j\right)$ such that $e^{g\left(i\right)}=f_0|_{U_i}$, but then 
	$Df_0|_{U_i}=dg\left(i\right)$ so $\chec^1\left(\mathcal{U};\Z\right)\subset\Ima C_\mathcal{U}$. To get the result, simply note 
	that $\chec^1\left(\mathcal{U};\Z\right)\simeq \CH^1\left(M;\Z\right)$ since $\mathcal{U}$ is a good cover.
\end{proof}
\begin{remark}
	Note that the arguments in the proof of proposition \ref{thing1} above are simply the standard considerations when one notices the fact that $D$ agrees locally with the differential of the logarithm and one notices the fact that 
	$C_{\mathcal{U}}\circ D$ factors through the 
	non-vanishing functions (by noting that we have the isomorphism $\mathcal{Z}^\times\simeq C^\infty\left(M\right)^\times\times\left(1+\hbar C^\infty\left(M\right)\hhbar\right)$).
\end{remark}

\begin{lemma}\label{P1P2P3}
	There are maps 
	\[P_1\colon\chec^1\left(\mathcal{U};\mathcal{Z}^\times\right)\stackrel{\sim}{\longrightarrow}\chec^1\left(\mathcal{U};C^\infty\left(M\right)^\times\right),\]
	\[P_2\colon\chec^1\left(\mathcal{U};{\AH}^\times\right)\longrightarrow\chec^1\left(\mathcal{U};C^\infty\left(M\right)^\times\right)\] 
	and 
	\[P_3\colon\chec^1\left(\mathcal{U};1+\hbar\C\hhbar\right)\stackrel{\sim}{\longrightarrow}\hbar\chec^1\left(\mathcal{U};\C\right)\hhbar,\]
	where $P_2$ has trivial kernel (note that as its domain is not necessarily a group this does not imply that the map is injective) and $P_1$ and $P_2$ are induced by the map 
	$f_0+\hbar f_1+\ldots\mapsto f_0$. 
\end{lemma}

\begin{proof}
	
	\leavevmode

	\noindent  The proof for $P_1$ is analogous to the proof for $P_2$ if not simply easier. So we will explicitly show the proof only for $P_2$. Consider the decreasing filtration 
	given by $({\AH}^\times)_n=1+\hbar^n\AH$ for $n>0$ and $({\AH}^\times)_0={\AH}^\times$. Then
	$({\AH}^\times)_{n+1}\lhd ({\AH}^\times)_n$ and $({\AH}^\times)_n\star({\AH}^\times)_m\subset({\AH}^\times)_{\min(n,m)}$. 
	So we have the 
	short exact sequence of (sheaves of) groups 
	\begin{equation}\label{filtseq1}1\rightarrow\left({\AH}^\times\right)_1\longrightarrow{\AH}^\times\longrightarrow C^\infty\left(M\right)^\times\rightarrow 1.\end{equation}
	Note that, if $f,g\in({\AH}^\times)_n$, then $f\star g=1+\hbar^n(f_n+g_n)\mod ({\AH}^\times)_{n+1}$ so we also have the short exact sequences of
	(sheaves of) groups
	\begin{equation}\label{filtseq2}1\rightarrow\left({\AH}^\times\right)_{n+1}\longrightarrow\left({\AH}^\times\right)_n\longrightarrow C^\infty\left(M\right)\rightarrow 0,\end{equation}
	for all $n\in\N$. The map $P_2$ is induced in the long exact sequence corresponding to \ref{filtseq1}. In order to show that it has trivial kernel, we should show that 
	$\chec^1(\mathcal{U};({\AH}^\times)_1)$ vanishes. Note 
	first of all 
	that, since $\chec^1(\mathcal{U};C^\infty(M))=0$, we find surjections 
	$\chec^1(\mathcal{U};({\AH}^\times)_{n+1})\twoheadrightarrow\chec^1(\mathcal{U};({\AH}^\times)_n)$ (which in fact have trivial kernel). 
	Suppose $S\colon J^2\rightarrow({\AH}^\times)_1$ is a cocycle. 
	Then, by the surjection above, $\exists a_1\colon J\rightarrow({\AH}^\times)_1$ 
	such that the cochain $a_1\cdot S\colon J^2\rightarrow({\AH}^\times)_2$ given by \[a_1\cdot S\left(i,j\right)=a_1\left(i\right)\star S\left(i,j\right)\star a_1\left(j\right)^{-1}\] is a 
	cocycle. Iterating this process yields a sequence $a_k\colon J\rightarrow({\AH}^\times)_k$ such that, 
	if we denote \[b_k:=\displaystyle\bigstarm_{j=1}^ka_j:=a_k\star a_{k-1}\star\ldots\star a_1,\] then $b_k\cdot S:J^2\rightarrow ({\AH}^\times)_{k+1}$ is a cocycle (cohomologous to $S$). Let us 
	denote $S_k:=b_k\cdot S$. Then we might consider its values 
	in $\AH$. Suppose $k\geq l\in\N$, we have
	\[\norm{S_k\left(i,j\right)-S_l\left(i,j\right)}=\norm{\left(b_k\left(i\right)-b_l\left(i\right)\right)\star S\left(i,j\right)\star b_k\left(j\right)^{-1}+b_l\left(i\right)\star 
		S(i,j) \star(b_k(j)^{-1}- b_l(j)^{-1})}\] \[\leq \max\left\{\norm{b_k(i)-b_l(i)}, 
	\norm{b_k(j)^{-1}-b_l(j)^{-1}}\right\}=\] \[=
	\max\left\{\norm{\left(-1+\bigstarm_{\alpha=l+1}^k a_\alpha\left(i\right)\right)\star b_l\left(i\right)}, 
	\norm{b_l\left(j\right)^{-1}\left(-1+\left(\bigstarm_{\alpha=l+1}^k a_\alpha\left(j\right)\right)^{-1}\right)}\right\}\leq 2^{-1-l}.\]
	This means the sequence $S_k\left(i,j\right)$ is Cauchy in $\AH$ (which is complete) and therefore has a limit $S_\infty\left(i,j\right)\in\AH$ for all 
	$\left(i,j\right)\in J^2$. Now note that since all the $S_k$ have values in $({\AH}^\times)_1$ so does $S_\infty$.
	Similarly we can show that the sequence $b_k$ has a limit $b\colon J\rightarrow ({\AH}^\times)_1$. Moreover, by the same computation, we see that $S_\infty=b\cdot S$. In particular, 
	$S$ is 
	cohomologous to $S_\infty$ and, since $\varprojlim \left({\AH}^\times\right)_n=1$, 
	we see that $[S]=[1]\in\chec^1(\mathcal{U};({\AH}^\times)_1)$. Since we started with an arbitrary cocycle this shows that the last cohomology pointed 
	set is in fact trivial. This shows the claim about $P_2$ and an easier 
	version of the same argument shows that $P_1$ is injective (since this is a group homomorphism). Since 
	$C^\infty\left(M\right)\hhbar^\times\simeq C^\infty\left(M\right)^\times\times\left(1+\hbar C^\infty\left(M\right)\hhbar\right)$, in the obvious way, we see that 
	$P_1$ is also surjective. 
	
	\vspace{0.3cm}
	
	The  map $P_3$ is simply induced by $\exp$ from lemma \ref{expLog}. 
	
\end{proof}
\begin{remark} To show lemma \ref{P1P2P3} we have used a method to pass from cohomology of $\AH^\times$ to cohomology of $C^\infty(M)^\times$ by showing that 
	$\chec^1\left(\mathcal{U};\left(\AH^\times\right)_1\right)=\{1\}$, using the fact that, by existence of a smooth partition of unity,  $\chec^1\left(\mathcal{U};C^\infty(M)\right)=0$. This method works equally well for group cohomology when 
	$\CH^1\left(\Gamma;C^\infty(M))\right)=0$ and we will use it in the next section. When we do this we will simply refer to the proof of lemma \ref{P1P2P3}, instead of basically repeating the proof. Note that, in general, 
	the method provides an indication of how the cohomology of $\AH^\times$ and $C^\infty(M)\hhbar^\times$ are both given by infinitely many copies of the cohomology of $C^\infty(M)$ and the 
	cohomology of $C^\infty(M)^\times$. 
\end{remark}

\begin{proposition}\label{brunt} 
	$\D$ is surjective on total sections. 
\end{proposition}

\begin{proof}
	
	\leavevmode
	
	\noindent In order to show the proposition, we will show that 
	\[\frac{\Ima \D\diagup\Ima D}{\chec^1\left(\mathcal{U};1+\hbar\C\hhbar\right)}\simeq \frac{\CH^1\left(M;\C\right)}{\CH^1\left(M;\Z\right)}.\] Then, a triple application of the 
	five lemma 
	will show that $\Ima \D =Z^1\left(M\right)\hhbar$. The inclusion of ${\chec^1\left(\mathcal{U};1+\hbar\C\hhbar\right)}$ in $\frac{\Ima\D}{\Ima D}$ should be evident from 
	the proof.
	
	\vspace{0.3cm} 
	
	Note that, if $g\in\G$ and $U\subset M$ is a coordinate neighborhood (or any other neighborhood such that $\CH^1 \left(U\right)=0$), then $g|_U=f\cdot k$ with $f\in\mathcal{Z}^\times|_U$ 
	and $k\in\Ker\nabla^\times|_U\simeq{\AH}^\times|_U$ by point (v) of lemma \ref{Gstuff}. 
	Now let $H\colon\Ima \D\rightarrow \chec^1\left(\mathcal{U};\C\hhbar^\times\right)$ be the map given by $H\left(\D g\right)=[\eta]$ where 
	$\eta\left(i,j\right)=\frac{f\left(i\right)}{f\left(j\right)}$ with the $f(i)$ given by the decompositions $g|_{U_i}=f\left(i\right)\cdot k\left(i\right)$. Again we leave it to the reader 
	to check that 
	this map is well-defined. 
	
	\vspace{0.3cm} 
	
	Now suppose $\D g\in\Ker H$ then $g|_{U_i}=f\left(i\right)\cdot k\left(i\right)$ and $\frac{f\left(i\right)}{f\left(j\right)}=\frac{c\left(j\right)}{c\left(i\right)}$ with 
	$c\left(i\right)\in\C\hhbar^\times$ for all $i$. This means that $g|_{U_i}=f\left(i\right)c\left(i\right)\cdot
	\frac{k\left(i\right)}{c\left(i\right)}$ and $\frac{f\left(i\right)c\left(i\right)}{f\left(j\right)c\left(j\right)}=1$, so $\exists fc\in\mathcal{Z}^\times$ such that
	$fc|_{U_i}=f\left(i\right)c\left(i\right)$. On the other hand we have that $k\left(i\right)\cdot k\left(j\right)^{-1}=\frac{c\left(i\right)}{c\left(j\right)}$, so $\exists \frac{k}{c}\in{\AH}^\times$
	such that $\frac{k}{c}|_{U_i}=\frac{k\left(i\right)}{c\left(i\right)}$. This shows that $g=fc\cdot\frac{k}{c}$ and thus $\Ker H=\Ima D$. So we conclude that 
	\[\frac{\Ima \D}{\Ima D}=\frac{\Ima \D}{\Ker H}\simeq \Ima H\subset \chec^1\left(\mathcal{U};\C\hhbar^\times\right).\] 
	
	\vspace{0.3cm}
	
	Now suppose $[\lambda]\in \chec^1\left(\mathcal{U};\C\hhbar^\times\right)$ such that $\exists f\in\check{\mathrm{C}}^0\left(\mathcal{U};\mathcal{Z}^\times\right)$ and 
	$k\in\check{\mathrm{C}}^0\left(\mathcal{U};{\AH}^\times\right)$ such that 
	$\frac{f\left(i\right)}{f\left(j\right)}=\lambda\left(i,j\right)=k\left(j\right)k\left(i\right)^{-1}$ for all $\left(i,j\right)\in J^2$. Then 
	$\frac{f\left(i\right)}{f\left(j\right)}k\left(i\right)k\left(j\right)^{-1}=\frac{\lambda\left(i,j\right)}{\lambda\left(i,j\right)}=1$ for all $\left(i,j\right)\in J^2$. So 
	$\exists g\in \G$ such that 
	$g|_{U_i}=f\left(i\right)\cdot k\left(i\right)$ and $H\left(\D g\right)=[\lambda]$. Conversely, if $[\lambda]=H\left(\D g\right)$ for some $g\in\G$, then obviously such 
	$0$-cochains exist. So we find that $\Ima H=\Ker I\cap \Ker Y$, for 
	$I$ and $Y$ the maps induced by the inclusions 
	\[{\AH}^\times\stackrel{Y}{\hookleftarrow}\C\hhbar^\times\stackrel{I}{\hookrightarrow} \mathcal{Z}^\times.\]

	\vspace{0.3cm}
	
	Now, by the lemma \ref{P1P2P3}, we see that $\Ker I=\Ker P_1\circ I$ and $\Ker Y=\Ker P_2 \circ Y$ and in fact these maps agree, i.e. $R:=P_1\circ I=P_2\circ Y$. They agree since they are all simply the map induced by 
	\[\begin{matrix} \C\hhbar^\times&\rightarrow&C^\infty\left(M\right)^\times\\
	c&\mapsto&c_0
	\end{matrix}.\]
	
	This map factors like $R=\partial\circ P$, where $P$ is induced by the projection $\C\hhbar^\times\twoheadrightarrow\C^\times$ and $\partial$ is induced by the inclusion of 
	$\C^\times$ in $C^\infty\left(M\right)^\times$. 
	Now, since $\C\hhbar^\times\simeq \C^\times\times\left(1+\hbar\C\hhbar\right)$ in the obvious way, we see that $\chec^1\left(\mathcal{U};1+\hbar\C\hhbar\right)=\Ker P$. So, we find that  
	\[\frac{\Ima H}{\Ker P}=\frac{\Ker R}{\Ker P}\simeq \Ker\partial.\]
	If we identify $\chec^1\left(\mathcal{U};C^\infty\left(M\right)^\times\right)\simeq \chec^2\left(\mathcal{U};\Z\right)$ using the exponential sequence \eqref{expsmooth},
	then we see that 
	$\partial$ is simply the connecting map in the exponential  sequence \eqref{expconstant} and this shows that $\Ker\partial\simeq\frac{\chec^1\left(\mathcal{U};\C\right)}{\chec^1\left(\mathcal{U};\Z\right)}$, which proves the claim. 
\end{proof}
In light of various exact sequences that appear in the following section, we will be able to use the following proposition. 

\begin{proposition}\label{H2} 
	Let \[0\rightarrow A\stackrel{i}{\longrightarrow} E \stackrel{\pi}{\longrightarrow} G\rightarrow 1\] be a $\Gamma$-equivariant \emph{central} extension, then we also have the exact sequence 
	\[0\rightarrow A^\Gamma\longrightarrow E^\Gamma\longrightarrow G^\Gamma\longrightarrow \CH^1(\Gamma;A)\longrightarrow  \CH^1(\Gamma;E)\longrightarrow \CH^1(\Gamma;G)\longrightarrow \CH^2(\Gamma;A).\] 
\end{proposition}
\begin{proof} 
	
	\leavevmode
	
	\noindent As mentioned we only need to extend the sequence to include $\CH^2(\Gamma;A)$ (see section 2.7 of \cite{NAC}). We simply do the usual (Abelian) construction and show that it still works 
	because the extension is central. So suppose $[\eta]\in \CH^1(\Gamma;G)$ and let $E_\gamma\mapsto \eta_\gamma$ for all $\gamma\in\Gamma$. Then $(\delta E)_{\gamma,\mu}\in\Ker\pi$ for all 
	$\gamma,\mu\in\Gamma$, so $a:=\delta E\colon \Gamma^2\rightarrow A$. By  writing out and noting that the computation takes place in the Abelian group $A$ we find 
	\[(\delta a)_{\gamma,\mu,\chi}=(\delta E)^{-1}_{\gamma\mu,\chi}\gamma((\delta E)_{\mu,\chi})(\delta E)_{\gamma,\mu\chi}(\delta E)^{-1}_{\gamma,\mu}= 
	E_{\gamma\mu\chi}E^{-1}_{\gamma\mu}(\delta E)^{-1}_{\gamma,\mu}\gamma(E_\mu)E_\gamma E^{-1}_{\gamma\mu\chi}=1
	\]
	So that $[a]\in \CH^2(\Gamma;A)$ and we will need to show that this map $[\eta]\mapsto [a]$ is well-defined. First suppose that $\tilde{E}_\gamma$ is another lift of $\eta_\gamma$ for each 
	$\gamma\in\Gamma$. Then note that $\tilde{E}_\gamma E_{\gamma}^{-1}$ and $\tilde{E}^{-1}_\gamma E_\gamma$ are in $A$ and thus central for all $\gamma\in\Gamma$, so we have 
	\[(\delta \tilde{E})_{\gamma,\mu}(\delta E)^{-1}_{\gamma,\mu}=(\delta \tilde{E}E)_{\gamma,\mu}\hspace{0.3cm}\forall \gamma,\mu\in \Gamma\]
	where $(\tilde{E}E)_\gamma=\tilde{E}_\gamma E_\gamma$. So both lifts will define cohomologous cocycles in $\CH^2(\Gamma;A)$. Now suppose $E'_\gamma$ lifts $\gamma(x)\eta_\gamma x^{-1}$ 
	for all $\gamma\in\Gamma$ and some $x\in G$. Then, if $\pi(X)=x$, we note that the element $(X\cdot E')_\gamma:=\gamma(X)^{-1}E_\gamma X$ lifts $\eta_\gamma$ for all $\gamma\in\Gamma$. Then note that 
	\[(\delta(X\cdot E'))_{\gamma,\mu}=\Ad (\gamma\mu)(X^{-1})((\delta E)_{\gamma,\mu})=(\delta E)_{\gamma,\mu}\hspace{0.3cm}\forall \gamma,\mu\in \Gamma,\] 
	which shows, when combined with the previous remark (concerning $\tilde{E}$), that the implied map in cohomology $\CH^1(\Gamma;G)\rightarrow \CH^2(\Gamma;A)$ is well-defined. 
	Lastly we should show that the sequence is exact at $\CH^1(\Gamma;G)$. So suppose $[\eta]\in \CH^1(\Gamma;G)$ maps to $0$ under the map described above. Then there is a lift $E_\gamma$ for 
	every $\eta_\gamma$ for all $\gamma\in\Gamma$ such that $(\delta E)_{\gamma,\mu}=(\delta \alpha)_{\gamma,\mu}$ for some $\alpha\colon\Gamma\rightarrow A$ and all $\gamma,\mu\in \Gamma$. 
	Then consider $E\alpha^{-1}$ given by $(E\alpha^{-1})_\gamma=E_\gamma\alpha^{-1}_\gamma$ and note that 
	\[(\delta E\alpha^{-1})_{\gamma,\mu}=(\delta E)_{\gamma,\mu}(\delta \alpha)_{\gamma,\mu}^{-1}=1\hspace{0.3cm}\forall \gamma,\mu\in \Gamma^2,\] 
	but also $E\alpha^{-1}\mapsto \eta$ under $\pi$. This shows that $[\eta]$ is in the image of $\CH^1(\Gamma;E)\rightarrow \CH^1(\Gamma;G)$. 
\end{proof}

\subsection{Computations}\label{5.3.3} 

Let us show in this section how the tools developed in the previous section can be applied to computation. In particular we will show that, under sufficiently restrictive assumptions on the cohomology of the 
group $\Gamma$ and the first cohomology of the manifold $M$, we find that there is a unique extension of the action to any deformation quantization. We will also show that 
in general the extensions are  not unique even when the first cohomology of the manifold vanishes and the group acts faithfully. Most of the computations here are done by applying 
group cohomology to the diagram \eqref{computingsquare}. 

\begin{proposition}\label{casefsphere}
	Suppose $\CH^1(M)=0$ and $|\Gamma|<\infty$,  then we have $\CH^1(\Gamma;\GZ)=\{1\}$. 
\end{proposition}
\begin{proof} 
	
	\leavevmode
	
	\noindent By vanishing of $\CH^1(M)$ and lemma \ref{Gstuff} we have the exact sequence 
	\begin{equation}\label{H10seq}1\rightarrow \Cht\longrightarrow\AH^\times\longrightarrow \GZ\rightarrow 1,\end{equation}
	since $\G=\cZ^\times\cdot\AH^\times$ and $\cZ^\times\cap \AH^\times=\Cht$.
	
	\vspace{0.3cm}
	
	The proof proceeds in two steps. First we show that the connecting map 
	\[\CH^1\left(\Gamma;\GZ\right)\rightarrow \CH^2\left(\Gamma;\Ch^\times\right)\] is trivial. 
	Secondly we show that the map \[\CH^1\left(\Gamma;\Ch^\times\right)\rightarrow\CH^1\left(\Gamma;\AH^\times\right)\] is surjective. Given these two facts, the proposition 
	is implied by proposition \ref{H2}. In the following we shall denote the differential of the cochain $c$ by $\delta c$. 
	
	\vspace{0.3cm}
	
	Suppose $\eta\colon\Gamma\rightarrow\GZ$ is a cocycle. We will first show that the image of the class of $\eta$ under the connecting map can be represented by a
	cocycle $c$ with values in $\C^\times$. So suppose  $E\colon\Gamma\rightarrow \AH^\times$ lifts $\eta$. By proposition \ref{H2} we find that $\delta E$ is a $2$-cocycle in $\Cht$ for the trivial action. Then, by the fact that $\CH^2(\Gamma;\C)=0$ and an analogous 
	argument to the one used in the proof of lemma \ref{P1P2P3}, we find $\lambda\colon\Gamma\rightarrow \Cht$ such that $(\delta\lambda)(\delta E)$ has values in $\C^\times$. Now note that, if $\tilde{E}$ denotes the cochain given by 
	$\tilde{E}_\gamma=\lambda_\gamma E_\gamma$ for all $\gamma\in\Gamma$, then $\tilde{E}$ also lifts $\eta$ and $\delta\tilde{E}=(\delta\lambda)(\delta E)$ (since $\Cht$ is central in $\AH^\times$). But then, if we write 
	$\tilde{E}=\sum \hbar^kE_k$, we find that 
	\[(\delta\tilde{E})_{\gamma,\mu}=(\delta E_0)_{\gamma,\mu}+\hbar S_{\gamma,\mu},\] 
	for all $\gamma,\mu\in\Gamma$, here we consider $\delta E_0$ as the boundary of the cochain $E_0\colon\Gamma\rightarrow C^\infty(M)^\times$. Now the fact that 
	$\delta\tilde{E}_{\gamma,\mu}\in\C^\times$ implies that $S=0$. So we find that $\delta\tilde{E}=\delta E_0$ as cochains with values in $\C^\times$. So, we 
	have found the cocycle $c:=\delta E_0$ representing the image of the class of $\eta$ under the connecting map. Note that, since $\CH^1(M;\Z)=0$, the inclusion of $\C$ in $C^\infty(M)$ and the exponential sequences \eqref{expconstant} and \eqref{expsmooth} induce the 
	commuting diagram with exact rows
	\[
	\begin{tikzpicture} 
	\draw node at (0,0) {$\CH^2(\Gamma;\Z)$};
	\path [draw,->] (.7,0) -- (2.2,0);
	\path [draw,->] (0,-.3) -- (0,-1.7);
	\draw node at (3,0) {$\CH^2(\Gamma;\C)$}; 
	\path [draw,->] (3.8,0) -- (5.1,0);
	\path [draw,->] (3,-.3) -- (3,-1.6);
	\draw node at (6,0) {$\CH^2(\Gamma;\C^\times)$}; 
	\path [draw,->] (6.9,0) -- (8.2,0);
	\path [draw,->] (6,-.3) -- (6,-1.7);
	\draw node at (9,0) {$\CH^3(\Gamma;\Z)$}; 
	\path [draw,->] (9.7,0) -- (11.2,0);
	\path [draw,->] (9,-.3) -- (9,-1.7);
	\draw node at (12,0) {$\CH^3(\Gamma;\C)$}; 
	\path [draw,->] (12,-.3) -- (12,-1.7);
	\draw node at (0,-2) {$\CH^2(\Gamma;\Z)$}; 
	\path [draw,->] (.7,-2) -- (2,-2);
	\draw node at (3,-2) {$\CH^2(\Gamma;C^\infty)$}; 
	\path [draw,->] (3.9,-2) -- (4.8,-2);
	\draw node at (6,-2) {$\CH^2(\Gamma;(C^\infty)^\times)$}; 
	\path [draw,->] (7.2,-2) -- (8.2,-2);
	\draw node at (9,-2) {$\CH^3(\Gamma;\Z)$}; 
	\path [draw,->] (9.7,-2) -- (11,-2);
	\draw node at (12,-2) {$\CH^3(\Gamma;C^\infty)$};
	\end{tikzpicture},
	\]
	where we have abbreviated $C^\infty:=C^\infty(M)$. 
	So, we have $\CH^2(\Gamma;\C^\times)\simeq \CH^2(\Gamma;C^\infty(M)^\times)$ by the five lemma, but then $[c]$ is trivial, since it is trivial in 
	$\CH^2(\Gamma;C^\infty(M)^\times)$ by construction. So
	the map $\CH^1(\Gamma;\GZ)\rightarrow \CH^2(\Gamma;\Cht)$ is the zero map. 
	
	\vspace{0.3cm} 
	
	Now suppose $\eta\colon\Gamma\rightarrow\AH^\times$ is a cocycle. Then, writing $\eta=\sum\hbar^k\eta_k$, we find that \\$\eta_0\colon\Gamma\rightarrow C^\infty(M)^\times$ is also a cocycle. 
	As above, we find the commuting diagram with exact rows 
	\[
	\begin{tikzpicture} 
	\draw node at (0,0) {$\CH^1(\Gamma;\Z)$};
	\path [draw,->] (.7,0) -- (2.2,0);
	\path [draw,->] (0,-.3) -- (0,-1.7);
	\draw node at (3,0) {$\CH^1(\Gamma;\C)$}; 
	\path [draw,->] (3.8,0) -- (5.1,0);
	\path [draw,->] (3,-.3) -- (3,-1.6);
	\draw node at (6,0) {$\CH^1(\Gamma;\C^\times)$}; 
	\path [draw,->] (6.9,0) -- (8.2,0);
	\path [draw,->] (6,-.3) -- (6,-1.7);
	\draw node at (9,0) {$\CH^2(\Gamma;\Z)$}; 
	\path [draw,->] (9.7,0) -- (11.2,0);
	\path [draw,->] (9,-.3) -- (9,-1.7);
	\draw node at (12,0) {$\CH^2(\Gamma;\C)$}; 
	\path [draw,->] (12,-.3) -- (12,-1.7);
	\draw node at (0,-2) {$\CH^1(\Gamma;\Z)$}; 
	\path [draw,->] (.7,-2) -- (2,-2);
	\draw node at (3,-2) {$\CH^1(\Gamma;C^\infty)$}; 
	\path [draw,->] (3.9,-2) -- (4.8,-2);
	\draw node at (6,-2) {$\CH^1(\Gamma;(C^\infty)^\times)$}; 
	\path [draw,->] (7.2,-2) -- (8.2,-2);
	\draw node at (9,-2) {$\CH^2(\Gamma;\Z)$}; 
	\path [draw,->] (9.7,-2) -- (11,-2);
	\draw node at (12,-2) {$\CH^3(\Gamma;C^\infty)$};
	\end{tikzpicture}.
	\]
	So, again by the five lemma, we find $f\in C^\infty(M)^\times$ such that $(\delta f)\eta_0$ has values in $\C^\times$. Denote by $\tilde{\eta}$ the cocycle given by 
	$\tilde{\eta}_\gamma=\gamma^*(f)\eta_\gamma f^{-1}$, then \[\tilde{\eta}((\delta f)\eta_0)^{-1}=1+\hbar S\] for some 
	$S\colon\Gamma\rightarrow\AH$. So, using the exact sequences 
	\[1\rightarrow 1+\hbar^{k+1}\AH\longrightarrow 1+\hbar^k\AH\longrightarrow C^\infty(M)\rightarrow 0,\] 
	the fact that $\CH^1(\Gamma;C^\infty(M))=0$ and the same reasoning as in lemma \ref{P1P2P3}, we find that there exists $P\in 1+\hbar\AH$ such that 
	$\alpha_\gamma(P)(1+\hbar S_\gamma)P^{-1}=1$ for all $\gamma\in\Gamma$. Thus we find that 
	\[[\eta]=[\tilde{\eta}]=[(\delta f)\eta_0]\in \CH^1(\Gamma,\AH^\times)\] which means that the map 
	$\CH^1(\Gamma,\Cht)\rightarrow \CH^1(\Gamma;\AH^\times)$ is surjective. So, by the sequence implied by proposition \ref{H2} and \eqref{H10seq}, we find that 
	$\CH^1(\Gamma;\GZ)=\{1\}$. \end{proof}

\begin{corollary} 
	For the conditions in proposition \ref{casefsphere} there is, up to conjugation by a fixed automorphism, a \emph{unique} extension of the group action to \emph{any} deformation quantization 
	by proposition \ref{casefsphere} and  corollary \ref{finiteistrivial}.
\end{corollary}

Note that proposition \ref{casefsphere} in particular contains the cases of the actions of $\Z/n\Z, D_{2n}, A_4, S_4$ and $A_5$ (finite subgroups of $SO(3)$) by symplectomorphisms on $S^2$. 
\begin{remark} 
	
	Note that the actual properties of the group that were used in the proof of proposition \ref{casefsphere} are 
	\begin{itemize}
		\item $\CH^i(\Gamma,\C)=\CH^i(\Gamma,C^\infty(M))=0$ for $i=1,2$ (for the trivial action on $\C$)
		\item $\CH^3(\Gamma,\C)\hookrightarrow \CH^3(\Gamma,C^\infty(M))$ (induced by the inclusion),
	\end{itemize}
	all of which are satisfied by finite groups. 
\end{remark}
\begin{remark}
	Let us consider for a moment the group $\Gamma$ acting trivially on the manifold $M$. Then of course we can lift this action to the trivial action on any deformation quantization and 
	the classification comes down to finding representations of $\Gamma$ as self gauge equivalences of the deformation. In this case we find that 
	\[\CH^1(\Gamma;\GZ)\simeq \bigslant{\Hom(\Gamma^{op},\GZ)}{\GZ}\] 
	where the quotient is taken with respect to the inclusion of $\GZ$ in $\Aut(\GZ)$ as inner automorphisms. It is to be expected that this yields non-equivalent extensions of the 
	trivial action of $\Gamma$. However, they would be rather pathological examples. Let us show that there may be non-equivalent extensions of the group action even when this action 
	$\Gamma\rightarrow\Symp(M)$ is injective. 
\end{remark}
\begin{example}\label{Zaction}
	Consider the action of $\Z$ on $S^2$ by rotation through an irrational angle $\theta$ around the vertical axis. Note that, since $\CH^2(S^2)\hhbar=\C\hhbar[\alpha]$ where $\alpha$ denotes 
	the standard symplectic 
	structure on $S^2$, we find that any action by symplectomorphisms that preserves an affine connection extends to any deformation quantization by proposition \ref{trivialextensions}. In particular, by uniqueness of the Levi-Civita connection, any action by symplectic isometries lifts to any deformation quantization. Note that
	we have $\CH^1(\Z;G)\simeq G/\sim$, where $g\sim h$ if $1(b)gb^{-1}=h$, whenever $\Z$ acts on any group $G$. Since we have $\CH^2(\Z;\C\hhbar^\times)=0$  and \eqref{H10seq}, we find 
	the 
	sequence
	\begin{equation}\label{seqS2irr}1\rightarrow \C\hhbar^\times\longrightarrow\left(\AH^\times\right)^\Z\longrightarrow \GZ^\Z\longrightarrow \C\hhbar^\times\longrightarrow 
	\AH^\times/{\sim}\longrightarrow \GZ/{\sim}\rightarrow 0.
	\end{equation}
	
	Suppose $c,c'\in\C\hhbar^\times$ such that their images in $\AH^\times/{\sim}$ coincide, i.e. there exists $g\in \AH^\times$ such that 
	$cr_\theta^*(g)=c'g$, where we denote the rotation inducing the action of $\Z$ by $r_\theta$ . Then note that, since the action has a fixed point at the north pole, we can evaluate $g$ there 
	to find that $c=c'$. So we find that the map $\C\hhbar^\times\rightarrow \AH^\times/{\sim}$ is injective. Now suppose $f\in\AH^\times$ such that there is $c\in\C\hhbar^\times$ and 
	$c\sim f$. Then, if we write $c=\sum \hbar^kc_k$ and $f=\sum \hbar^kf_k$, we find that there exists some $g_0\in C^\infty(S^2)^\times$ such that 
	$c_0r_\theta^*(g_0)=f_0g_0$ (for the undeformed product). Now note that this implies that $f_0$ must take the value $c_0$ at both the north and south pole. So, we find that the map 
	$\C\hhbar^\times\rightarrow \AH^\times/{\sim}$ is definitely not surjective, since there exist many non-vanishing functions on $S^2$ that do not have the same value at the north and 
	south pole. Then exactness of \eqref{seqS2irr} shows that $\CH^1(\Z;\GZ)\simeq \GZ/{\sim}\neq \{1\}$ and so there exist multiple extensions of this action to any deformation quantization. 
\end{example}

\begin{corollary} 
	Suppose $\Gamma=\Z$ acts on $(M,\omega)$ with $\CH^1(M)=0$ and there is more than $1$ fixed point, then $\CH^1(\Gamma;\GZ)\neq\{1\}$. 
\end{corollary}

Finally let us show that the diagram \eqref{computingsquare} allows us to make conclusions about $\CH^1(\Gamma;\GZ)$ even in the case that $\CH^1(M)\neq 0$. 

\begin{example}\label{T2irr}
	Consider the action of $\Z$ on $\T^2=\T\times\T$ by irrational rotation of one of the coordinates. Then we find that the induced action on $T^1_\hbar(\mathbb{T}^2)$ is trivial and so 
	\[\CH^1(\Z;T^1_\hbar(\mathbb{T}^2))=T^1_\hbar(\mathbb{T}^2)\] 
	by the discussion in the beginning of example \ref{Zaction}. Suppose $0\neq l\in T^1_\hbar(\mathbb{T}^2)$, then we can lift this to $[0]\neq [z]\in \CH^1(\Z;Z^1(M)\hhbar)\simeq Z^1(M)\hhbar/{\sim}$, 
	since $\CH^2(\Z;\mathcal{Z}^\times/\C\hhbar^\times)=0$. By proposition \ref{brunt} we can find $g\in\G$ such that 
	$\mathbb{D}g=z$. In particular $g$ represents a non-trivial class in $\CH^1(\Z;\G)$. Its image $\overline{g}$ in $\CH^1(\Z;\GZ)$ must, by commutativity of \eqref{computingsquare}, be mapped 
	to the non-trivial class $l$ and therefore it cannot be trivial. So we find that $\CH^1(\Z;\GZ)\neq \{1\}$. In fact this shows that 
	\[T^1_\hbar(\mathbb{T}^2)\hookrightarrow \CH^1(\Z;\GZ).\] Note also that, since 
	$\CH^2(\Z;\mathcal{Z}^\times)=0$, we find that the map $\CH^1(\Z;\G)\rightarrow \CH^1(\Z;\GZ)$ is surjective and, since $\CH^2(\Z;\C\hhbar^\times)=0$, so is the map 
	$\CH^1(\Z;\AH^\times)\rightarrow \CH^1(\Z;\AH^\times/\C\hhbar^\times)$. Thus 
	we find that $\CH^1(\Z;\GZ)$ is essentially given by $T^1_\hbar(\mathbb{T}^2)$ and $\CH^1(\Z;\AH^\times)$ the last of which is given essentially by 
	$\CH^1(\Z;C^\infty(\mathbb{T}^2))$ and $\CH^1(\Z;C^\infty(\mathbb{T}^2)^\times)$.
\end{example}

Many more examples can of course be considered, for instance when $\Gamma$ is finite, but $\CH^1(M)\neq 0$. We will leave these examples to the reader.

\end{document}